\documentclass{amsart}
\usepackage{amsmath,amsthm,amssymb,latexsym,graphicx,verbatim,enumerate,%
	ifpdf,mdwlist}
\ifpdf
\usepackage{hyperref}%[unicode]{hyperref}
\else
\usepackage[hypertex]{hyperref}
\fi

\usepackage{amssymb, amsmath, amsthm}
\usepackage{color}

\renewcommand\setminus{\smallsetminus}

\providecommand{\M}{\mathcal{M}}
\providecommand{\N}{\mathbb{N}}

\newcommand{\C}{\mathcal{C}}
\renewcommand{\L}{\mathcal{L}}
\renewcommand{\M}{\mathcal{M}}
\renewcommand{\N}{\mathcal{N}}

\providecommand{\C}{\mathcal{C}}
\providecommand{\set}[1]{\lbrace #1 \rbrace}

\providecommand{\abs}[1]{\lvert#1\rvert}

\providecommand{\Th}[1]{ {\rm Th} (#1)}

\newcommand\w{\omega}
\newcommand\Sum{\Sigma}

\DeclareMathOperator{\acl}{acl}

\DeclareMathOperator{\SRM}{SRM}

\DeclareMathOperator{\tp}{tp}
\DeclareMathOperator{\tprqf}{tp_{r.q.f.}}
\DeclareMathOperator{\arity}{arity}

\newtheorem{thm}{Theorem}[section]
\newtheorem{defn}[thm]{Definition}
\newtheorem{lem}[thm]{Lemma}
\newtheorem{obs}[thm]{Observation}

\newtheorem{cory}[thm]{Corollary}

\newtheorem*{claim*}{Claim}
\theoremstyle{definition}
\newtheorem*{remark*}{Remark}

\newcommand{\dalet}{\ensuremath{%
\mathchoice{\includegraphics[height=2ex]{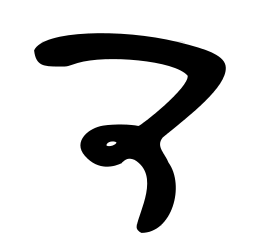}}
{\includegraphics[height=2ex]{daled}}
{\includegraphics[height=1.5ex]{daled}}
{\includegraphics[height=1ex]{daled}}
}}

\newcommand{\setcol}[2]{\{#1 \mid #2\}}
\newcommand{\strong}{\leq}

\title[${[0,n]}\cup \{\omega\}$ is a spectrum]{$[0,n]\cup \{\omega\}$ is a spectrum of a non-disintegrated flat strongly minimal model complete theory in a language with finite signature}
\author{Uri Andrews, Omer Mermelstein}
\address{Department of Mathematics, University of Wisconsin--Madison, 480 Lincoln Dr., Madison, WI 35706, USA}
\email{andrews@math.wisc.edu, omer@math.wisc.edu}
\thanks{The first author was partially supported by NSF grant DMS-1600228.}

\keywords{Spectrum of recursive models, Spectrum of computable models, Flatness, Hrushovski construction}
\subjclass[2020]{03C57, 03D47, 03C30}

\begin{document}

\begin{abstract}
	We build a new spectrum of recursive models ($\SRM(T)$) of a strongly minimal theory. This theory is non-disintegrated, flat, model complete, and in a language with a finite signature.
\end{abstract}

\maketitle
\section{Introduction}

The countable models of uncountably categorical theories were characterized by Baldwin and Lachlan \cite{BL} as being completely determined by a single dimension. Thus these models are very well understood model-theoretically. We seek to also understand them recursion-theoretically. A fundamental question is which of these models have recursive presentations.

From Baldwin and Lachlan's characterization in terms of dimensions, the countable models of any uncountably categorical but not countably categorical theory form an elementary chain $M_0\prec M_1\prec M_2\prec\cdots\prec M_\omega$. For such a theory $T$, the Spectrum of Recursive Models of $T$ ($\SRM(T)$) is the set of $i$ so that $M_i$ has a recursive presentation. The spectrum problem asks which subsets of $\omega+1$ appear as $\SRM(T)$ for some theory $T$.

The following theorem summarizes all the positive results currently known about the existence of spectra.

\begin{thm}\label{knownSpectra}
	The following are spectra of recursive models of strongly minimal theories:
	\begin{itemize}
		\item $\emptyset$
		\item $[0,\w]$, i.e., $\w+1$
		\item $\{0\}$ (Goncharov~\cite{Go78})
		\item $[0,n]$, i.e., $\{0,\ldots n\}$, for any $n \in \w$
		(Kudaibergenov~\cite{Ku80})
		\item $[0,\w)$, i.e.,~$\w$ (Khoussainov/Nies/Shore~\cite{KNS97})
		\item $[1,\w]$, i.e., $(\w+1)\smallsetminus\{0\}$
		(Khoussainov/Nies/Shore~\cite{KNS97}, see also \cite{KSS07})
		\item $\{1\}$ (Nies~\cite{Ni99})
		\item $[1,\alpha)$ for any $\alpha \in [2,\w]$ (Nies/Hirschfeldt, see
		Nies~\cite[p.~314]{Ni99})
		\item $\{\w\}$ (Hirschfeldt/Khoussainov/Semukhin~\cite{HKS06})
		\item $\{0,\w\}$ (Andrews~\cite{A0w})
	\end{itemize}
\end{thm}

While relatively few sets are known to be spectra, the only known upper bound on all spectra is that every spectrum must be $\Sigma^0_{\omega+3}$, and for a model complete theory, $\SRM(T)$ is $\Sigma^0_4$ \cite{Ni99}.

In the hopes of coming to a better understanding of possible spectra, a few approaches have been taken, including: Focusing on the strongly minimal theories, focusing on theories with particular geometric properties, and focusing on theories in languages with finite signatures. In these cases, we have better characterizations of the possible spectra.

By ``particular geometric properties'', we refer to the Zilber trichotomy: Zilber conjectured that every strongly minimal theory is either disintegrated ($\acl(A)=\bigcup_{a\in A}\acl(a)$) or locally modular (after adding one constant, we get $\dim(A\cup B)=\dim(A)+\dim(B)-\dim(A\cap B)$ for any finite-dimensional closed sets $A$ and $B$) or is field-like (there is an interpretable an infinite field with no definable sets on it aside from the ones definable in the field itself).

Under such assumptions, we can completely characterize the possible spectra:

\begin{thm}[Andrews-Medvedev \cite{AM}]
	If $T$ is disintegrated strongly minimal and the language has a finite signature, then $\SRM(T)=\emptyset,\{0\},$ or $[0,\omega]$.	
\end{thm}

By Herwig-Lempp-Ziegler \cite{HLZ}, all three of these cases are in fact spectra of disintegrated theories in languages with finite signature.

\begin{thm}[Andrews-Medvedev \cite{AM}]
	If $T$ is a modular strongly minimal theory expanding a group in a language with finite signature, then $\SRM(T)=\emptyset,\{0\},$ or $[0,\omega]$.
	
	If $T$ is a field-like strongly minimal theory expanding a field in a language with finite signature, then $\SRM(T)= [0,\omega]$.
	
\end{thm}

Thus, in each prototypical case of the Zilber trichotomy, if the language has a finite signature, then there are very few possible spectra. In particular, either all or no models of positive dimension have recursive presentations.

Hrushovski \cite{Udi} showed that the Zilber conjecture is false and produced a new class of strongly minimal sets, all of which have a geometric property called flatness. Formally, flatness is defined below, but intuitively it describes dimension as being purely combinatorial in a way that allows for no algebraic rules (such as associativity of a group operation) to hold:

\begin{defn}
	A theory is flat if whenever $\{E_i \mid i\in I\}$ is a finite collection of finite-dimensional closed sets, and $s$ ranges over the subsets of $I$, we have $\Sum_{s}(-1)^{\abs{s}}\dim(E_s)\leq 0$.
\end{defn}

Andrews \cite{A0n}\cite{A0w} showed that $[0,n]$, $[0,\omega)$, and $\{\omega\}$ are spectra of recursive models of flat strongly minimal theories in languages with finite signature. Thus, it is possible for a strongly minimal theory in a finite language to have some positive-dimensional models be recursive and others not.  

In this paper, we present another schema of spectra of a flat strongly minimal theory in a language with finite signature: $[0,n]\cup \{\omega\}$

We also point out the most important technical innovation in this paper. For the general technique, we code extra relations in a structure by the number of extensions of a certain type over a base. That is, in the amalgamation, we allow more or fewer occurrences of a certain extension in order to code information. This only successfully codes that information if a tuple will have the maximal number of extensions allowed. In condition (3'') on page \pageref{3''}, we see that if the base is strong enough, then it has the maximal number of realizations, but there is much room for exceptions. In the previous uses of this technique (\cite{A0n} and \cite{A0w}), there were two different tricks used to avoid these exceptions, but one only works if the size of the base of the extension is at most one more than the dimension of the prime model and the other only works if we have a bound on the size of the extensions that we will need. Neither can work in our current setting, and they are both fragile methods. In section \ref{Technicals} we present the correct solution for this problem: We present a collection of ``unblockable'' extensions which for any base whatsoever in any Hrushovski construction must have the maximal possible number of extensions. This tool should make any combination of recursion theory with Hrushovski constructions far easier in the future.

A different approach to recursion-theoretically understanding strongly minimal theories is to consider the relative complexity of models. That is, if one model of a theory $T$ is recursive, how difficult can it be to compute the other models of $T$? In the case of a disintegrated theory, every other model must be recursive in $\mathbf{0}''$ \cite{GHLLM} and there is a theory where $\mathbf{0}''$ is needed to compute the other models \cite{KLLS}. In general, the degrees which compute every countable model of a strongly minimal theory which has a recursive model are exactly the degrees which are high over $\mathbf{0}''$ \cite{ALS,AK18}. That is, exactly the degrees $\mathbf{d}$ so that $\mathbf{d} >\mathbf{0}''$ and $\mathbf{d}'\geq \mathbf{0}^{(4)}$.

\section{Background}

\subsection{Notation}

We write $\exists^{k}\bar{x} \phi(\bar{x})$ to mean that there are at least $k$ \emph{disjoint} tuples $\bar{x}$ which satisfy $\phi$. 

\subsection{Hrushovski constructions in infinite languages where $\mu$ depends on the self-sufficient closure}

Fix $\L$ a relational language with the relations indexed by a (finite or not) initial segment of $\omega$. We will enforce in our construction that each relation is symmetric (if we do not do this, the same construction works with no changes -- this is purely a stylistic choice).  

The following definitions and lemmas are standard to all Hrushovski constructions. 

\begin{defn}
	The pre-dimension function on $\L$-structures is the function $\delta$ from finite $\L$-structures to $\mathbb{Z}\cup \{-\infty\}$ so that $\delta(A)=\abs{A}-\Sigma_{R\in \L}\#R(A)$. Here $\#R(A)$ counts the number of occurrences of the relations on $A$ (counting $R(\bar{a})$ together with $R(\sigma(\bar{a}))$ as a single relation for all permutations $\sigma$).

	For $A, B\subseteq C$ with $A,B$ finite:
	\begin{itemize}
		\item We write $\delta(B/A)=\delta(A\cup B)-\delta(A)$.
		\item We define $\delta(A,C)=\inf\{\delta(D)\mid A\subseteq D\subseteq C, D\text{ finite}\}$.
		\item We say $A$ is strong (also called self-sufficient) in $C$, written $A\leq C$, if $\delta(A)=\delta(A,C)$.
		\item We say $B$ is simply algebraic over $A$ if $A\cap B=\emptyset$, $A\leq A\cup B$, $\delta(B/A)=0$, and there is no proper non-empty subset $B'$ of $B$ so that $\delta(B'/A)=0$.
		\item We say $B$ is minimally simply algebraic over $A$ if $B$ is simply algebraic over $A$ and $A$ is minimal so that $B$ is simply algebraic over $A$.
		\item If $A\subseteq B$ and $B\smallsetminus A$ is (minimally) simply algebraic over $A$, then we say $A\subseteq B$ (or $B/A$) is a (minimally) simply algebraic extension.
	\end{itemize}

	Let $\C_0$ be the collection of $\L$-structures $C$ so that $\delta(A)\geq 0$ for every finite $A\subseteq C$.
\end{defn}

\begin{obs}
	For $A,B$ finite subsets of an $\L$-structure $C$,
	$$\delta(A\cup B)\leq \delta(A)+\delta(B)-\delta(A\cap B)$$ with equality if and only if there are no relations holding between $A$ and $B$ other than those inside $A$ or those inside $B$. In this case, we say $A$ and $B$ are freely joined over $A\cap B$ and we write $A\cup B = A\oplus_{A\cap B}B$.
\end{obs}
\begin{proof}
	This is just inclusion-exclusion on the number of relations holding in $A\cup B$.
	
\end{proof}

The following two lemmas capture some basic facts about the notion of strong substructure.

\begin{lem}
	If $A\leq B$, then for any $X\subseteq B$, $\delta(X\cap A)\leq \delta(X)$.
	
	If $A\leq B\leq C$, then $A\leq C$.
\end{lem}
\begin{proof}
	Suppose $A\leq B$. Then $\delta(A)\leq \delta(X\cup A)\leq \delta(X)+\delta(A)-\delta(X\cap A)$. Thus, $\delta(X\cap A)\leq \delta(X)$. 
	
	Now suppose $A\leq B\leq C$. Let $A\subseteq X\subseteq C$. Then $\delta(X\cap B)\leq \delta(X)$ since $B\leq C$, and $\delta(A)=\delta(X\cap B\cap A)\leq \delta(X\cap B)$ since $A\leq B$. So, $\delta(A)\leq \delta(X)$.
\end{proof}

\begin{lem}\label{disjointnessfromstrength}
	Suppose $X\leq C$ and $Y,Z\subseteq C$ are distinct, simply algebraic over $X$. Then $Y$ and $Z$ are disjoint.
\end{lem}
\begin{proof}
	Since $\delta(X\cup Y \cup Z)\leq \delta(X\cup Y)+\delta(X\cup Z)-\delta(X\cup (Y\cap Z))$, we can subtract $\delta(X)$ from both sides and see that $\delta(Y\cup Z/X)\leq \delta(Y/X)+\delta(Z/X)-\delta(Y\cap Z/X)=-\delta(Y\cap Z/X)$. Since $Y$ and $Z$ are simply algebraic over $X$, either $Y\cap Z=\emptyset$ or $\delta(Y\cap Z/X)>0$. But since $X\leq C$, $0\leq \delta(Y\cup Z/X)=-\delta(Y\cap Z/X)$. So $Y$ and $Z$ are disjoint.
\end{proof}

\begin{defn}
	For $C\in \C_0$ and $A\subseteq C$ finite, the self-sufficient closure of $A$ is the smallest set $X$ so that $A\subseteq X\leq C$.
\end{defn}

\begin{lem}\label{ssclosurealg}
	For any finite $A\subseteq C\in \C_0$, the self-sufficient closure exists, is unique, and is finite.
\end{lem}
\begin{proof}
	Take $X\supseteq A$ with minimal $\delta(X)$ and take $X$ minimal as such (i.e. it has no proper subset containing $A$ with the same value of $\delta$). The minimality of $\delta(X)$ implies that $X\leq C$. Suppose $X$ were not unique, then there would be another such set $Y\supseteq A$ with $\delta(Y)=\delta(X)$. Then $\delta(X\cup Y)\leq \delta(X)+\delta(Y)-\delta(X\cap Y)<\delta(Y)=\delta(X)$ by minimality of the set $X$. But this would contradict the minimality of $\delta(X)$. 
\end{proof}

The following definitions are necessary to do the Hrushovski construction with an infinite signature. We will build our theory by using an infinite signature and then taking a reduct to a finite sub-signature.

\begin{defn}\label{def:oftheform}
	For any $\L$-structures $\bar{a},\bar{b}\subseteq C$,  the relative quantifier-free type of $\bar{b}$ over $\bar{a}$, written $\tprqf(\bar{b}/\bar{a})$, is the set of formulas $\{R(\bar{x}_i,\bar{y}_i)\mid (\bar{b}_i\bar{a}_i)\subseteq (\bar{b}\cup \bar{a})^{\arity(R)}\smallsetminus \bar{a}^{\arity(R)}, R\in \L, C\models R(\bar{b}_i,\bar{a}_i) \}$.
	
	Fix $\mu(A,B,m)$ to be a function that takes in pairs of $\L$-structures so that $A\subset B$ is a minimally simply algebraic extension, and a number $m\in \omega\cup\{\infty\}$, and $\mu$ outputs a number in $\omega$ so that $\mu(A,B,m)\geq \delta(A)$.
	
	We also require $\mu$ to satisfy: For every relative quantifier-free type $\Psi$ of a minimally simply algebraic extension, there is a finite sublanguage $\L'\subseteq \L$ so that $\tp_{q.f.}(A)\vert \L'=\tp_{q.f.}(A')\vert \L'$ and $\tprqf(B/A)=\tprqf(B'/A')=\Psi$ implies $\mu(A,B,m)=\mu(A',B',m)$. Further, for every $A,B$, we must have $\lim_{m\rightarrow \infty} \mu(A,B,m)=\mu(A,B,\infty)$.

	For any $A\subseteq C$, we let $g_C(A)$ be the least $m$ so that there exists an $X\subseteq C$ so that $A\not\leq X$ which is witnessed by using only the first $m$ relations in the language $\L$ and $\abs{X}\leq \abs{A}+m$. If there is no such $m$, then $A\leq C$ and we let $g_C(A)=\infty$.
	
	We have chosen to present the definition of $\mu$ in terms of an extension $A\subseteq B$. For $A,B\subseteq C$ with $A,B$ finite, if $B$ is minimally simply algebraic over $A$ (in particular, $B$ is disjoint from $A$), then we also write $\mu(A,B,m)$ for $\mu(A,A\cup B, m)$.

\end{defn}

The following observation is a critical fact about how the function $g$ behaves.
\begin{obs}\label{gSameStrong}
	If $A\subseteq B\leq C$, then $g_B(A)=g_C(A)$.
\end{obs}
\begin{proof}
	That $g_C(A)\leq g_B(A)$ is immediate since any $X\subseteq B$ so that $A\subseteq X$ and $\delta(X)<\delta(A)$ is also a subset of $C$. Suppose $X\subseteq C$ is so that $A\subseteq X$ and $\L_0\subseteq \L$ is a sublanguage so that restricting to this sublanguage, $\delta(X\vert \mathcal{L}_0)<\delta(A\vert\mathcal{L}_0)$. Then $\delta(X\cap B)\leq \delta(X)< \delta(A)$, so in the same sublanguage, $B$ contains a set no larger than $X$ witnessing that $A$ is not strong. So $g_B(A)\leq g_C(A)$.
\end{proof}

\begin{defn}
	Let $Y$ and $X$ be finite $\L$-structures so that $Y$ is minimally simply algebraic over $X$.
	
	We define $\L_{Y/X}$ to be the finite collection of symbols occurring in $\tprqf(Y/X)$.
	
	Suppose $B$ and $A$ are finite $\L$-structures such that $\tprqf(B/A) \supseteq \tprqf(Y/X)$ and $\tp_{q.f.}(X)=\tp_{q.f.}(A)$, then we say the extension $B$ over $A$ is of the form of $Y$ over $X$.
\end{defn}

\begin{defn}
	
	Let $\C_{\mu}$ be the collection of finite $\L$-structures $C \in \C_0$ that satisfy:
	
	Suppose $A,B^1,\ldots,B^r$ are disjoint subsets of $C$ so that each $B^i/A$ is of the form of $Y/X$, then $r\leq \mu(X,Y,g_C(A))$.
	
	For any $\forall$-axiomatizable elementary property $\zeta$ which is preserved by free joins, let $\C_{\mu}^{\zeta}$ be the collection of $C\in \C_{\mu}$ so that $\C\models \zeta$.
	
	Note that for trivial $\zeta$, $\C_{\mu}^\zeta = \C_{\mu}$.
\end{defn}

In most uses of the amalgamation method, we use the class $\C_{\mu}$, but it requires very little extra work to include the generality of working with $\C_{\mu}^\zeta$ and it will make our construction of a strongly minimal theory $T$ with $\SRM(T)=[0,n]\cup\{\omega\}$ slightly cleaner.

\begin{obs}\label{firstorderoftheform}
	Fix $Y/X$ a minimally simply algebraic extension. There is a first-order formula which is true in any $C\in \C^\zeta_\mu$ and implies that $C$ does not contain disjoint subsets  $A,B^1,\ldots,B^r$ of the form of $Y/X$ with $r>\mu(X,Y,g_C(A))$.
	
\end{obs}
\begin{proof}
	Let $\L'$ be the finite sublanguage of $\L$ guaranteed in Definition \ref{def:oftheform}, and let $m$ be an integer so that $\mu(X,Y,k)=\mu(X,Y,\infty)$ for any $k\geq m$.
	
	Let $\rho(A)$ say that $A\vert \L'\cong X\vert \L'$. For each $k\in \omega$, let $\psi_k(A,C^1,\ldots C^k)$ say that the sets are disjoint, and that $\tprqf(C^j/A)\supseteq \tprqf(Y/X)$ for each $j$. For each $l\leq m$, let $\phi_l$ say that $g_C(A)=l$. Finally, let $\theta$ be the formula which says $\forall A$, if $\rho(A)$ holds, then
	\[
	\bigvee_{l<m}\left( \phi_l(A)\wedge \neg \exists Z \psi_{\mu(X,Y,l)+1}(A,Z)\right)\vee \left(\bigwedge_{l<m}\neg \phi_l(A)\wedge \neg \exists Z \psi_{\mu(X,Y,\infty)+1}(A,Z)\right).
	\]
	Let $\Omega$ be the collection of all the extensions $Y'/X'$ so that $\tprqf(Y'/X')=\tprqf(Y/X)$ and $X'\vert \L'\cong X\vert \L'$. Then $\theta$ says that the $\mu$-bound is obeyed for each extension in $\Omega$. Thus $\theta$ is true in every $C\in \C^\zeta_{\mu}$ and implies that $C$ respects the $\mu$-bound for $Y/X$.
\end{proof}

The following observation follows directly from the definition of being of the form of $Y/X$.

\begin{obs}\label{removerelationsstayoftheform}
	Let $A\subseteq B$ be finite $\L$-structures. Let $A'\subseteq B'$ be formed by removing some occurrence $R(\bar{b})$ of some relation $R$ from $B$. If $B'/A'$ is of the form of $Y/X$, then either $B/A$ is of the form of $Y/X$ or $\bar{b}\subseteq A$.
\end{obs}

The following three Lemmas allow us to perform ``strong amalgamation'' on the class $\C^\zeta_{\mu}$, which will lead to a generic structure. The theory of this generic will be our strongly minimal theory.

\begin{lem}\label{Lemma42}
	Suppose $A,B_1,B_2\in \C_0$, $A=B_1\cap B_2$, and $A\leq B_1$. Let $E=B_1\oplus_A B_2$. Suppose $F,C^1,\ldots C^r$ are disjoint substructures of $E$ such that each $C^i$ is minimally simply algebraic over $F$. 
	Then one of the following holds:
	\begin{itemize}
		\item One of the $C^i$ is contained in $B_1\smallsetminus A$ and $F\subseteq A$.
		\item $F\cup \bigcup_{i\leq r} C^i$ is contained either entirely in $B_1$ or entirely in $B_2$.
		\item $r\leq \delta(F)$
		\item For one $C^i$, setting $Z=(F\cap A)\cup (C^i\cap B_2)$, $\delta(Z/Z\cap A)<0$. Further, one of the $C^j$ is entirely contained in $B_1\smallsetminus A$. (Note that this cannot happen if $A\leq B_2$.)
	\end{itemize}

\end{lem}
\begin{proof}
	A careful reading of Lemma 3 of \cite{Udi} will show that this is what is proved there. The full proof appears as Lemma 42 in \cite{AndrewsThesis}.
\end{proof}

\begin{lem}[Algebraic Amalgamation Lemma]
	Suppose $A=B_1\cap B_2$, $A,B_1,B_2\in \C^\zeta_\mu$, and $B_1\smallsetminus A$ is simply algebraic over $A$. Let $E$ be the free-join of $B_1$ with $B_2$ over $A$. Then $E\in \C^\zeta_\mu$ unless one of the following holds:
	\begin{enumerate}
		\item\label{case1} $B_1\smallsetminus A$ is minimally simply algebraic over $F\subseteq A$ and there are $\mu(F, B_1\smallsetminus A, g_{B_2}(F))$ disjoint extensions over $F$ of the form of $B_1\smallsetminus A$ over $F$. 
		\item \label{case2} There is a set $Y\subseteq B_2$ with $\abs{Y}\leq \abs{B_1\smallsetminus A}$ so that $\delta(Y\vert \L_{B_1/A}/A\vert \L_{B_1/A})<0$.
		\item\label{case3} There is a minimally simply algebraic extension $Y/X$ and a set $F\subseteq B_1$ and $C\subseteq B_1$ of the form of $Y/X$ so that $\mu(X,Y,g_{E}(F))<\mu(X,Y,g_{B_1}(F))$.
	\end{enumerate} 
\end{lem}
\begin{proof}
	It is immediate that $E\models \zeta$ because $\zeta$ is preserved under free joins. Let $X\subseteq E$. Then $\delta(X)=\delta(X\cap B_1)+\delta(X\cap B_2)-\delta(X\cap A)\geq \delta(X\cap B_2)\geq 0$, since $X$ is the free-join of $X\cap B_1$ with $X\cap B_2$ over $X\cap A$ and $A\leq B_1$. Thus $E\in \C_0$.
	
	Suppose that $Y/X$ is a minimally simply algebraic extension, and $F,C^1,\ldots,C^r$ are disjoint subsets of $E$ so that each $C^i/F$ is of the form of $Y/X$. We restrict $E$ to $\L_{Y/X}$ for all tuples outside of $F$. 
	
	By Lemma \ref{Lemma42}, there are 4 cases to consider:
	\begin{itemize}
		\item One of the $C^i$ is contained in $B_1\smallsetminus A$ and $F\subseteq A$. Since $B_1\smallsetminus A$ is  simply algebraic over $A$, we have that $C^i=B_1\smallsetminus A$. In this case, we have that $r$ is at most one more than the number of disjoint extensions over $F$ of the form of $Y/X$ in $B_2$. Since $B_2\leq E$, Observation \ref{gSameStrong} shows that  $g_{B_2}(F)=g_{E}(F)$. So, if $r>\mu(X, Y,g_E(F))$, then in $B_2$ we already have $g_{B_2}(F)$ disjoint extensions over $F$ of the form of $Y/X$. Since $B_1\smallsetminus A/F$ is minimally simply algebraic and of the form of $Y/X$, each of these extensions is of the form of $B_1\smallsetminus A/F$. Thus we have (\ref{case1}) above.
		\item $F\cup\bigcup_{i\leq r}C^i$ is entirely contained in $B_1$ or $B_2$. If it's contained in $B_2$, then since $B_2\in \C_\mu$ and $B_2\leq E$, we see that $r\leq \mu(X,Y,g_{B_2}(F))=\mu(X,Y,g_{E}(F))$. If it's contained in $B_1$, then since $B_1\in \C_\mu$, we have that $r\leq \mu(X,Y,g_{B_1}(F))$. Either $r\leq \mu(X,Y,g_{B_1}(F))\leq \mu(X,Y,g_{E}(F))$ or we are in case (\ref{case3}) above.
		\item $r\leq \delta(F)$. Then it's automatic that $r\leq \mu(X,Y,g_E(F))$ as $\mu(X,Y,m)$ is always $\geq \delta(X)=\delta(F)$.
		\item For one $C^i$, setting $Z=(F\cap A)\cup (C^i\cap B_2)$, we see $\delta(Z/Z\cap A)<0$. Let $Y=Z\smallsetminus A$. Then $\delta(Y/A)<0$. Further, one of the $C^j$ is contained in $B_1\smallsetminus A$, so $\abs{Y}\leq \abs{C^i\cap B_2}\leq \abs{B_1\smallsetminus A}$. Since one of the $C^j$ is contained in $B_1\smallsetminus A$, all of the relations between the set $Y$ and $A$ which were retained when we took the reduct above are in $\L_{B_1/A}$, so $\delta(Y\vert \L_{B_1/A}/A\vert \L_{B_1/A})<0$. Thus, case (\ref{case2}) holds.
	\end{itemize}
\end{proof}

\begin{lem}[Strong Amalgamation Lemma]
	Suppose $A,B_1,B_2\in \C^\zeta_\mu$ and $A\leq B_i$ for $i=1,2$. Then there exists a $D\in \C^\zeta_\mu$ so that $B_2\leq D$ and a $g:B_1\rightarrow D$ so that $g$ is the identity map on $A$ and $g(B_1)\leq D$.
\end{lem}
\begin{proof}
	We may assume that there is no $B'$ so that $A\leq B'\leq B_1$ as otherwise we can first amalgamate this $B'$ with $B_2$ over $A$. Thus, either $B_1$ is $A\cup \{x\}$ where $x$ is unrelated to any element in $A$ or $B_1$ is simply algebraic over $A$, say minimally simply algebraic over $F\subseteq A$. In the first case, the free-join suffices. In the second case, the free-join works unless one of the three conditions enumerated in the Algebraic Amalgamation Lemma holds. The second and third cannot hold, because $A\leq B_2$. Thus we can assume that $B_1\smallsetminus A$ is minimally simply algebraic over $F\subseteq A$ and that in $B_2$ there are $C^1,\ldots C^{r}$ which are $\mu(F,B_1\smallsetminus A,g_{B_2}(F))$ disjoint extensions of the form of $B_1\smallsetminus A/F$. Since $A\leq B_1$ and $A\leq B_2$, we have $g_A(F)=g_{B_1}(F)=g_{B_2}(F)$. Thus, it cannot be that all of these $C^j$ are contained in $A$, since then $B_1$ would have violated the $\mu$-bound. Without loss of generality, $C^1\not\subseteq A$. $C^1$ cannot be partially in $A$ since $A\leq B_2$. Thus $C^1\subseteq B_2\smallsetminus A$. Since $A\leq B_2$, there are no extra relations in $\tprqf(C^1/A)$ other than those in $\tprqf(B_1\smallsetminus A/F)$, and $A\cup C^1\strong B_2$. Thus, we can form $g$ by sending $B_1\smallsetminus A$ to $C^1$ over $A$.
\end{proof}

Using the Strong Amalgamation Lemma, we build a generic model $\M$. We discuss its theory in the next subsection.

\subsection{The Theory of the Generic}

Using the Strong Amalgamation Lemma, {via a Fra\"iss\'e-style construction,} we get a model $\M:=\M^\zeta_\mu$ (if $\zeta$ is trivial, we write $\M=\M_{\mu}$) which satisfies the following 3 properties:

\begin{enumerate}
	
	\item $\M$ is countable
	\item If $A\leq \M$ is finite, then $A\in \C^\zeta_\mu$.
	\item Suppose $B\leq \M$, $B\leq C$, and $C\in \C^\zeta_\mu$. Then there exists an embedding $f:C\rightarrow \M$ which is the identity on $B$ and $f(C)\leq \M$.
\end{enumerate}

By a standard back-and-forth on strong substructures, these three properties characterize $\M$ up to isomorphism. We want to show that $\M$ is saturated by showing that any countable elementary extension of $\M$ is isomorphic to $\M$. To do so, we must check that these properties are elementary. We consider the properties:

\begin{enumerate}
	\item[(2')] This is broken down into 3 statements:
	\begin{enumerate}[a)]
	\item 
	$\M\models \zeta$.
	\item
	$\delta(A)\geq 0$ for every finite $A\subseteq \M$.
	\item
	If $F,C^1,\ldots C^r$ are disjoint subsets of $\M$ and each $C^i/F$ is of the form of $Y/X$. Then $r\leq \mu(X,Y,g_{\M}(F))$.
	\end{enumerate}
	\item[(3')] There is an infinite set $I\subseteq \M$ on which no relation holds and every finite $A\subset I$ is strong in $\M$.
	\item[(3'')]
	\label{3''}
	Suppose $B\subseteq \M$, $B\leq C$, $C\in \C^\zeta_\mu$, and $C\smallsetminus B$ is simply algebraic over $B$, say minimally simply algebraic over $F\subseteq B$. Suppose that there is no $Y\subseteq \M$ such that $\delta(Y\vert \L_{C/B}/B\vert\L_{C/B})<0$ and $\abs{Y}\leq \abs{C\smallsetminus B}$. Further suppose that there is no minimally simply algebraic extension $H/G$ and sets $H',G'\subseteq C$ of the form of $H/G$ so that $\mu(G,H,g_{C}(G))\neq \mu(G,H,g_{\M\oplus_B C}(G))$. Then there are $\mu(F, C\smallsetminus B, g_{\M}(F))$ disjoint extensions over $F$ of the form of $C\smallsetminus B$ over $F$ in $\M$.

\end{enumerate}

\begin{lem}
\label{123equivprimes}
	(1),(2),(3) is equivalent to (1),(2'),(3'),(3'').
\end{lem}
\begin{proof}
	Suppose (1),(2),(3) are true. To see (2')(a) holds, since $\zeta$ is universal it suffices to check that $\zeta$ holds on every finite substructure of $\mathcal{M}$. Every finite $A\subseteq \mathcal{M}$ is contained in its finite self-sufficient closure $C$, which is strong in $\mathcal{M}$. As $C\in \mathcal{C}_\mu^\zeta$ by (2), we have $C\models\zeta$ and thus $A\models \zeta$. Similarly by $C\in \mathcal{C}_\mu^\zeta$, (2')(b) holds on $A$. Finally, to see (2')(c), let $A=F\cup\bigcup_{i\leq r}C^i$ and again consider the self-sufficient closure $C$ of $A$. Then since $C\in \mathcal{C}_\mu^\zeta$ by (2), we have that $r\leq \mu(X,Y,g_{C}(F))$, but $g_C(F)=g_{\mathcal{M}}(F)$ because $C\leq \mathcal{M}$. This shows that (2') holds. (3') and (3'') hold similarly by applying the algebraic amalgamation lemma, i.e., the algebraic amalgamation lemma shows that it would keep us in $\C^\zeta_\mu$ to have these sets, and property (3) then gives us the needed sets inside $\M$.

	Now we suppose (1),(2'),(3'),(3''). Suppose $C\leq \M$. Then for any set $A\subseteq C$, we have $g_C(A)=g_{\M}(A)$. Thus condition (2') ensures that $C\in \C^\zeta_\mu$, so (2) holds. Property (3) follows from (3') and (3'') exactly as the strong amalgamation lemma follows from the algebraic amalgamation lemma. 
\end{proof}

\begin{lem}\label{elementaryconditions}
	Conditions (2'),(3'') are elementary schemata, i.e., there is a set of sentences $\Psi$ so that $M\models \Psi$ if and only if $M$ satisfies the conditions. Furthermore, $\Psi$ can be chosen to be a set of $\forall\exists$-sentences. 
	
	Condition (3') is preserved in elementary extensions or substructures containing $I$. That is, if $I\subseteq M\preceq N$, then $I$ satisfies the condition in $M$ if and only if it satisfies the condition in $N$.
\end{lem}
\begin{proof}
	The first two conditions in (2') are easy to see are elementary. For the third, we use the formulas from Observation \ref{firstorderoftheform}. Carefully examining the formula $\theta$ produced in Observation \ref{firstorderoftheform}, one can see they are equivalent to $\forall\exists$ formulas. Alternatively, to see that (2') is equivalent to a collection of $\forall\exists$-sentences, it suffices to see that it is preserved in unions of chains. Suppose $M_0\subseteq M_1\subseteq \cdots \subseteq \bigcup_i M_i=N$ are so that each $M_i$ satisfies (2'). Note that $\lim_i g_{M_i}(A)=g_N(A)$ for every $A\subseteq N$. Suppose there are $F,C^1,\ldots C^r$ disjoint subsets of $N$ and each extension $C^j/F$ is of the form of $Y/X$. Then take $i$ large enough that $M_i$ contains $F\cup \bigcup_i C^i$ and $g_{M_i}(A)=g_{N}(A)$. Then since $M_i$ satisfies (2'), we see that $r\leq \mu(X,Y,g_{M_i}(A))=\mu(X,Y,g_N(A))$.
	
	For (3'), it suffices to note that a finite set being strong in the structure is defined by an infinite schema of universal sentences.
	
	For (3''), the fact that it is first-order to determine the value of $\mu(X,Y,m)$ (as in Observation \ref{firstorderoftheform}) and that it is first-order to be strong enough so that $\mu(G,H,g_C(G))=\mu(G,H,g_{M\oplus_B C}(G))$ suffices to make this first order. Let us see that it is preserved in unions of chains. Again suppose that $M_0\subseteq M_1\subseteq \cdots \subseteq \cup_i M_i=N$ and each $M_i$ satisfies (3''). Let $B\subseteq N$, $B,C\in \C^\zeta_\mu$ be as in the hypothesis of (3''). If $B$ is strong enough in $N$ that there is no $Y$ with $\abs{Y}<\abs{C\smallsetminus B}$ and $\delta(Y\vert \L_{C/B}/B\vert \L_{C/B})<0$, then this is true in every $M_i$ which contains $B$ as the non-existence of this $Y$ is a universal condition. Similarly, since for every $A\subseteq N$, we have that $\lim_i g_{M_i}(A)=g_N(A)$, and $M_i\leq M_i\oplus_B C$ and $N\leq N\oplus_B C$, we also have that $\lim_i g_{M_i\oplus_B C}(A)=g_{N\oplus_B C}(A)$. Thus if there is no $H/G$ an extension in $C$ so that $\mu(G,H,g_C(G))\neq \mu(G,H,g_{N\oplus_B C}(G))$, then the same is true in $M_i$ for any large enough $M_i$. Thus, in a large enough $M_i$, we must have $\mu(F,C\smallsetminus B,g_{M_i}(F))=\mu(F,C\smallsetminus B,g_{N}(F))$ disjoint extensions over $F$ of the form of $C\smallsetminus B$ over $F$. Thus, we have these extensions in $N$ as well.	
\end{proof}

\begin{lem}
	$\M $ is saturated.
\end{lem}
\begin{proof}
	Since any countable elementary extension of $\M$ is isomorphic with $\M$, we see that there are only countably many $m$-types in $\Th{\M}$ for each $m$. Thus, there is some countable saturated model of the theory of $\M$. But this, being an elementary extension of $\M$, is isomorphic with $\M$.
\end{proof}

Next we want to describe algebraicity inside $\M$.
\begin{defn}
	For any finite set $A\subset \M$, we define $d(A)=\delta(A,\M)$. Recall, $\delta(A,\M)=\min\{\delta(C)\mid A\subseteq C\subset \M, \text{ $C$ finite}\}$.
\end{defn}

\begin{lem}\label{d goes up is generic}
	If $d(\{x\}\cup A)=d(A)+1=d(\{y\}\cup A)$, then $(\M,Ax)\cong_A (\M,Ay)$.
\end{lem}
\begin{proof}
	Let $B$ be the self-sufficient closure of $A$, so $\delta(B)=d(A)$. Then $B\leq \M$. Thus $\{x\}\cup B\leq \M$ and $\{y\}\cup B\leq \M$. Using property (3) and a back-and-forth along strong substructures, we see that $(\M,Bx)\cong_B (\M,By)$, which implies the needed isomorphism.
\end{proof}

\begin{lem}\label{same d is algebraic}
	If $d(\{x\}\cup A)=d(A)$, then $x\in \acl(A)$.
\end{lem}
\begin{proof}
	Suppose $d(\{x\}\cup A)=d(A)$. Let $B$ be the self-sufficient closure of $A$, so $\delta(B)=d(A)$. Lemma \ref{ssclosurealg} shows that $B$ is algebraic over $A$.
	
	Now we show that $x\in \acl(B)$, which suffices since $B\subseteq \acl(A)$. Fix $E$ to be a set so that $\delta(E)=d(\set{x}\cup A)=d(A)$ and $\set{x}\cup A\subseteq E$. Then $\delta(E\cup B)\leq \delta(E)+\delta(B)-\delta(E\cap B)$. If $E$ does not contain $B$, then $\delta(E\cap B)>\delta(B)$ by minimality of $B$. Then $\delta(E\cup B)<\delta(E)=d(A)$, a contradiction. So $E$ contains $B$.
	
	Take a sequence of extensions $B=B_0\subset B_1\subset\cdots B_n=E$ with each $B_{i+1}$ chosen to be minimal containing $B_i$ contained in $E$ with $\delta(B_{i+1})=d(A)$. It follows from $\delta(B_{i})=d(A)$ that each $B_i\leq \M$. Then $B_{i+1}$ is a simply algebraic extension over $B_i$. Thus, the $\mu$-bound ensures that there are not infinitely many extensions over $B_i$ which are disjoint and of the form of $B_{i+1}/B_i$, and Lemma \ref{disjointnessfromstrength} shows that any two extensions of the form of $B_{i+1}/B_i$ must be disjoint. Thus $B_{i+1}$ is algebraic over $B_i$. Conclude that $E$ is algebraic over $B$.
\end{proof}

\begin{cory}\label{itisstronglyminimal}
	$\Th{\M}$ is strongly minimal.
\end{cory}
\begin{proof}
	In the previous two lemmas, we saw that over any set there is a unique non-algebraic type realized in $\M$. Since $\M$ is saturated, this implies that $\Th{\M}$ is strongly minimal.
\end{proof}

 \begin{lem}\label{modelcompletenessprereduct}
The axioms stating the model is infinite, (2') and (3'') axiomatize $\Th{\M^\zeta_\mu}$. Thus, $\Th{\M^\zeta_\mu}$ is model complete.
\end{lem}
\begin{proof}
Observe that Lemma 2.21 holds for any model satisfying (2').
Let $N$ be infinite and satisfy (2') and (3'').
Let $N'\succeq N$ contain a countable indiscernible sequence $I=(x_i)_{i\in \omega}$. We observe that $d(\{x_i\mid i<k\})=k$. Were this not the case, then we would have some $x_i\in \acl(\{x_j\mid j<i\})$ by Lemma 2.21, which cannot happen by indiscernibility. Thus every finite initial subsequence of $I$, and thus every finite subsequence of $I$ is strong in $N'$. Let $M\preceq N'$ be countable and contain $I$. Then $M$ satisfies (1), (2'), (3'), and (3''), with (3') witnessed by $I$. Thus $M\cong \mathcal{M}_\mu^\zeta$ by Lemma \ref{123equivprimes}. So $N\models \Th{\mathcal{M}_\mu^\zeta}$.
\\
By Lindstr\"om's test \cite[8.3.4]{BigHodges} and Lemmas \ref{elementaryconditions} and \ref{itisstronglyminimal}, $\Th{\M^\zeta_\mu}$ is model complete.
\end{proof}

\begin{cory}\label{firstTisflat}
$T$ is flat and non-disintegrated.
\end{cory}

\begin{proof}
As in \cite[Lemma 15]{Udi}, the geometry associated to any hypergraph via the $d$ function is flat. Lemmas \ref{d goes up is generic} and \ref{same d is algebraic} show that $d$ is the dimension function of the acl-geometry in $\M$.
\end{proof}

\section{Technical amalgamation facts}\label{Technicals}

In this section, we gather some facts about amalgamation that will be important in our particular construction of the theory whose spectrum of recursive models is $[0,n]\cup \{\omega\}$. In our language, we will have a ternary relation symbol. Thus below we will discuss hypergraphs with a ternary edge.

In the course of our construction, we will need two facts that this section will provide. Firstly, we will need, at one point in the construction of the recursive saturated model, to cause the dimension of a tuple to decrease while staying in the amalgamation class. In Corollary \ref{droppingdimensionsatwill}, we will do this by extending a finite structure in a non-strong way. Usually, Hrushovski amalgamation constructions only allow for strong amalgamations at any stage in the construction, so this requires some separate considerations, which we do in this section. Secondly, we will need a collection of extensions that are guaranteed to have the maximal possible number of occurrences over any base. We will call these unblockable extensions, and we show they exist in Corollary \ref{unblockablesexist}.

For the remainder of this section, we will work with $\C_{\mu}$. While it will be true that all of the particular structures that we mention in this section will satisfy the formula $\zeta$ that we use in our construction, we focus on only $\C_{\mu}$ in this section for the sake of generality and re-usability of these results. It is immediate that if the particular structures mentioned in this section satisfy $\zeta$, then our results in this section about $\C_{\mu}$ also hold for $\C^\zeta_{\mu}$.

We implicitly assume that $\L$ has a ternary relation symbol $R$, and that the isomorphism type of three elements with a unique ternary edge is an element of $\C_\mu$.

\begin{remark*}
In this section we do not use that $\mu(A,B,m)\geq \delta(A)$. Rather, we only use that $\mu$ is such that $\C_\mu$ satisfies the Strong Amalgamation Lemma.
\end{remark*}

\subsection{Dropping dimensions}
We introduce a particular type of extension $B$ over $A$ with $\delta(B/A)=-1$. We will use this below to decrease dimension by adding extensions of this type over particular tuples in our constructed structure.

\begin{defn}

	For $t>k\geq 2$, let $A=\set{a_1,\dots, a_k, g,h}$. For every $l>k$, by $a_l$ we mean $a_i$ where $i\equiv_k l$ and $1\leq i\leq k$. Let $B=B_1\cup B_2$ where $B_1 = \set{b_1,\dots, b_{t}, b_{t+1}}$, $B_2 = \set{b_{t+1},\dots, b_{2t}, b_1}$ and let $R = R_1\cup R_2\cup R_3$ where
	\begin{gather*}
	R_1 = \setcol{(b_i,a_i,b_{i+1})}{0< i \leq t}
	\\
	R_2 = \setcol{(b_i,a_i,b_{i+1})}{t< i< 2t} \cup\set{(b_{2t}, g, b_1)}
	\\
	R_3 = \set{(b_1,h,b_{t+1})}.
	\end{gather*}
	Let $D_t$ be the hypergraph with vertex set $A\cup B$ and edge set $R$.
\end{defn}

Here we classify the extensions $B_0$ over $A_0$ with $\delta(B_0/A_0)\leq 0$ where $B_0\subseteq B$ and $A_0\subseteq A$.

\begin{lem}
	\label{possible zeros in D_t}
	If $A_0\subseteq A$, $\emptyset\neq B_0\subseteq B$ are such that $\delta(B_0/A_0)\leq 0$, then at least one of the following holds:
	\begin{enumerate}
		\item
		$A_0 \supseteq A\setminus\set{h}$ and $B_0 = B$
		\item
		$A_0 \supseteq A\setminus\set{g}$ and $B_0 \supseteq B_1$
		\item
		$A_0 = A$ and $B_0 \supseteq B_2$
	\end{enumerate}
	Moreover, if $\delta(B_0/A_0)<0$, then $A_0=A$ and $B_0 = B$.
\end{lem}

\begin{proof}
	Since we are trying to show that $B_0$ is large (i.e. contains either $B$, $B_1$ or $B_2$, depending on the case), we can assume that there is no proper subset $B_0'\subset B_0$ which has $\delta(B_0'/A_0)\leq 0$. 
	For every $b_i\in B$ clearly $\delta(b_i/A_0) = 1$, so $|B_0|>1$. Then for every $b_i\in B_0$, the element $b_i$ must appear in at least two relations in $A_0\cup B_0$ or else $B_0\setminus\set{b_i}$ contradicts the minimality of $B_0$.
	Now, for every $i\notin\set{1,t+1}$, if $b_i\in B_0$, then $(b_{i-1},a_{i-1},b_i)$, $(b_i,a_i,b_{i+1})$ are edges in $A_0\cup B_0$ (if $i=2t$, then take $a_i = g$ and $b_{i+1} = b_1$). For $i\in\set{1,t+1}$, if $b_i\in B_0$, then at least one of $b_{i-1}$, $b_{i+1}$ (for $i=1$, take $b_{i-1}$ to be $b_{2t}$) must be an element of $B_0$. Hence, $R[A_0\cup B_0]$, the restriction of $R$ to $A_0\cup B_0$, must contain at least one of $R_1$, $R_2$. In particular, $b_1,b_{2t}\in B_0$. Since $t\geq k+1$, also in any case $A\setminus\set{g,h}\subseteq A_0$. Furthermore, if $R_2\subseteq R[A_0\cup B_0]$, then also $g\in A_0$. I.e., if $g\notin A_0$, then $R_1\subseteq R[A_0\cup B_0]$.
	
	If $h\notin A_0$, then as $b_1\in B_0$ and $b_1$ must appear in two relation, we have $b_2,b_{2t}\in B_0$, and consequently $R_1,R_2\subseteq R[A_0\cup B_0]$. This is case (1), so we may assume $h\in A_0$. Now, if $g\notin A_0$, then as we've seen we must be in case (2). Finally, assume we are not in cases (1) or (2), in particular $A=A_0$. Since $B_1\nsubseteq B_0$, also $R_1\nsubseteq R[A_0\cup B_0]$. Then $R_2\subseteq R[A_0\cup B_0]$ and we are in case (3).
	
	Now for the additional part, assume instead that $B_0$ is minimal such that $\delta(B_0/A_0)< 0 $. Again, if $b_i\in B_0$ then it must appear in at least two relations in $R[A_0\cup B_0]$. By what we've shown, for some $i\in \set{1,2}$, $B_i\subseteq B_0$. Observe that $\delta(B_i/A_0)\geq \delta(B_i/A) = 0$, so there exists $b_j\in B_0\setminus B_i$. As before, this implies that also $B_{3-i}\subseteq B_0$, i.e., $B_0 = B$. A short check yields that for any $A\setminus \set{g,h} \subseteq X \subset A$, we have $\delta(B/X)\geq 0$. So $A_0=A$ and indeed $\delta(B/A) < 0$.
\end{proof}

Now we show that $D_t$ is in $\C_{\mu}$. All we are assuming about $\mu$ is that the isomorphism type of three elements with a unique ternary edge is an element of $\C_{\mu}$.

\begin{lem}
	\label{if edges than D_t is classy}
	$D_t\in \C_{\mu}$.
\end{lem}

\begin{proof}
	First we show that no subset of $D_t$ has $\delta(X)<2$ unless $\abs{X}\leq 1$. Assume to the contrary that $|X|\geq 2$, $\delta(X) < 2$. Then there must be edges in $X$, hence $X\cap A$ and $X\cap B$ are non-empty. $\delta(D_t) > 1$, so $X\neq D_t$. Then by Lemma \ref{possible zeros in D_t}, $\delta(X\cap B/X\cap A) \geq 0$. So, $1-\delta(X\cap A) \geq \delta(X)-\delta(X\cap A)\geq 0$ which means we must have $1\geq \delta(X\cap A) =|X\cap A|\geq 1$ since there are no relations holding in the set $A$. But then again by Lemma \ref{possible zeros in D_t}, $\delta(X) >\delta(X\cap A) = 1$ in contradiction to $\delta(X)<2$.
	
	We show $D_t$ embeds strongly into $\M_\mu$. Embed $\set{a_1,\dots, a_k, b_1}$ onto an independent subset of $\M_\mu$. Because the isomorphism type of a ternary edge is in $\C_\mu$, by genericity of $\M_\mu$, any two independent points extend to an edge. Embed $b_2$ onto an extension of $b_{1}, a_1$ to an edge. Now embed $b_3$ onto an extension of $b_2, a_2$ to an edge. Continue in this manner until $b_{2t}$ has been embedded. Now embed $g$ onto an extension of $b_{2t}, b_1$ to an edge and $h$ onto an extension of $b_1, b_{t+1}$ to an edge. Observe that each embedded point must be new. Now, since $D_t\strong\M_{\mu}$, we have $D_t\in \C_{\mu}$ by property (2) of $\M_{\mu}$.
\end{proof}

\begin{obs}
	\label{obs: points in msa are in relations}
	Let $C$ be minimally simply algebraic over $F$. 
	From the definition of a minimally simply algebraic extension, it follows that every element of $F$ appears in at least one edge in $F\cup C$ that is not an edge of $F$, and every element of $C$ appears in at least two edges in $F\cup C$ (unless $\abs{C}=1$).
\end{obs}

\begin{obs}
	\label{obs: if F is not in B_2 then r is restricted by valency}
	Suppose $E=B_1\oplus_A B_2$ where every point in $B_1\setminus A$ appears in at most $k$ edges in $B_1$. Let $F, C^1,\dots, C^r\subseteq E$ be disjoint such that each $C^i/F$ is a minimally simply algebraic extension. If $F\nsubseteq B_2$, then it is immediate from Observation \ref{obs: points in msa are in relations} that $r\leq k$.
\end{obs}

\begin{defn}
	Say that $\mu$ is \emph{$k$-permissive} if $\mu(A,B,m)\geq k$ whenever $B$ is minimally simply algebraic over $A$ and $m\in \omega\cup \{\infty\} $.
\end{defn}

\begin{lem}
	\label{t big enough to stay classy}

	Suppose $\mu$ is 3-permissive and $M\in \C_{\mu}$. Fix $\bar{c}\in M^{k+2}$ with $k\geq 2$, $\delta(\bar{c},M)>0$. Label $\bar{c}=(a_1,\ldots, a_k,g,h)$. 
	For each $t$, let $E_t$ be the free join of $M$ and $D_t$ over $\bar{c}$. Then $E_t\in\C_\mu$ for all large enough $t$.
\end{lem}

\begin{proof}
	We start with the following claim:
	
	\begin{claim*}
		If $A\subseteq M$ and $Z\subseteq E_t$ contains $A$ such that $\delta(Z/A)<0$, then either $\delta(Z\cap M/A)<0$ or $|Z\cap (D_t\setminus \bar{c})|\geq t$.
	\end{claim*}
	\begin{proof}
		Using $\delta(Z)=\delta(Z\cap M)+\delta(Z\cap D_t)-\delta(Z\cap \bar{c})$, we subtract $\delta(A)$ from both sides and combine with $\delta(Z\cap D_t)-\delta(Z\cap \bar{c})=\delta(Z\cap D_t/Z\cap \bar{c})$ to see $0> \delta(Z/A)=\delta(Z\cap M/A)+\delta(Z\cap D_t/Z\cap \bar{c})$. Thus, either $\delta(Z\cap M/A)<0$ or $\delta(Z\cap D_t/Z\cap \bar{c})<0$. In the latter case, Lemma \ref{possible zeros in D_t} shows that $|Z\cap (D_t\setminus \bar{c})|\geq t$.
	\end{proof}

	Choose $t>|M|$ greater than the index of any relation symbol appearing in $M$, in particular greater than $\max\setcol{g_M(A)}{A\not\strong M}$, and also large enough so that for every $A, B\subseteq M$ such that $B/A$ is of the form of an extension $Y/X$, $\mu(X,Y,m) = \mu(X,Y,\infty)$ for every $m\geq t$. Denote $E:=E_t$. By the claim, this guarantees that for any $A\subseteq M$, if $A\not\strong M$, then $g_E(A) = g_M(A)$. If $A\strong M$, then $\mu(A,B,g_E(A)) = \mu(A,B,\infty)$ whenever $B\subseteq M$ is minimally simply algebraic over $A$. In any case, whenever $A,B\subseteq M$ are such that $A/B$ is of the form of an extension $Y/X$, then $\mu(X,Y,g_E(A)) = \mu(X,Y, g_M(A))$.
	
	Now assume for a contradiction that $F, C^1,\dots, C^r\subseteq E$ are disjoint such that each $C^i/F$ is of the form of a minimally simply algebraic extension $Y/X$, and $r>\mu(X, Y, g_E(F))$. By construction, every point in $D_t\setminus X$ is in at most three relations in $E_t$, so 3-permissiveness and Observation \ref{obs: if F is not in B_2 then r is restricted by valency} imply $F\subseteq M$. By what we've shown, at most $\mu(X, Y, g_E(F)) = \mu(X, Y, g_M(F))$ of the $C^i$ are contained in $M$. Without loss of generality assume $C^1\nsubseteq M$. Then by $\delta(C^1\cap D_t/F\cup (C^1\cap M)) \leq 0$, Lemma \ref{possible zeros in D_t} implies $|C^1|> t> |M|$ and so none of the $C^i$ can be fully contained in $M$. Using the same reasoning for $\delta(C^2/F\cup (C^2\cap M))\leq 0$, again by Lemma \ref{possible zeros in D_t}, we see that $C_1$ and $C_2$ intersect inside $D_t$, in contradiction.
\end{proof}

\begin{lem}
	\label{geometry changes the way we want it to change}
	Let $M$, $\bar{c}$ and $E_t$ be as in the lemma above and let $X\subseteq M$. Then $\delta(X, E_t) = \delta(X, M)$, unless there is some $Y\subseteq M$ containing $\bar{c}\cup X$
	with $\delta(Y) = \delta(X, M)$, in which case $\delta(X, E_t) = \delta(X, M)-1$.
\end{lem}

\begin{proof}
	Clearly $\delta(X, E_t) \leq \delta(X, M)$. Let $X\subseteq Z\subseteq E_t$, then
	\[
	\delta(Z) = \delta(Z\cap M) + \delta(Z\cap D_t/Z\cap \bar{c}).
	\]
	The first summand is at least $\delta(X, M)$ and the second summand is at least $-1$. Thus, to witness $\delta(X,E_t)< \delta(X,M)$, we must have that $\delta(Z\cap M)=\delta(X, M)$ and $\delta(Z\cap D_t/Z\cap F) = -1$. By Lemma \ref{possible zeros in D_t}, this means $D_t\subseteq Z$, and in particular $\bar{c}\subseteq Z\cap M$. Conversely, if $Y$ such as in the statement exists then $Y\cup D_t$ witnesses $\delta(X,E_t) \leq \delta(X,M)-1$.
\end{proof}

\begin{cory}\label{droppingdimensionsatwill}	
	Let $\mu$ be 3-permissive, and $A\in \C_{\mu}$, $\bar{b}\in A$ so that $\delta(\bar{b},A)>0$. Then there exists $B\in \C_{\mu}$ containing $A$ so that for every $X\subseteq A$, $\delta(X,B)=\delta(X,A)$ unless there is $Y\supseteq X$ so that $\delta(Y)=\delta(X,A)$ and $\bar{b}\in Y$, in which case $\delta(X,B)=\delta(X,A)-1$.
\end{cory}

\begin{proof}
	If $|\bar{b}|\geq 4$, then by labeling $\bar{b} = (a_1,\dots, a_{|\bar{b}|-2}, g,h)$ Lemma \ref{t big enough to stay classy} and Lemma \ref{geometry changes the way we want it to change} give us the desired $B$.
	
	Otherwise, construct $M'$ by adding new points $p_1,\dots, p_{4-|b|}$ to $M$ and no new edges. Let $\bar{b}'$ be the concatenation of $\bar{b}$ with $(p_1,\dots p_{4-|\bar{b}|})$. Now apply the above to $M$ and $\bar{b}'$ consecutively $5-|b|$ times. To use Lemma \ref{geometry changes the way we want it to change} note that after the penultimate application, every set containing $\bar{b}$ can be extended to a set containing $\bar{b}'$ without changing its $\delta$ value.
\end{proof}

\subsection{Unblockability}\label{unblockthis}

\begin{defn}
	We say that a minimally simply algebraic extension $X\subseteq Y$ is 
	$k$-\emph{unblockable} if for any $k$-permissive $\mu$, if $Y\in C_\mu$ then for any $X\subseteq Z\in C_\mu$, either $Y\oplus_X Z\in C_\mu$ or $Z$ already contains $\mu(X,Y,g_Z(X))$ disjoint extensions of the form of $Y/X$ over $X$.
\end{defn}

\begin{obs}
A minimally simply algebraic extension $X\subseteq Y$ is $k$-unblockable if and only if for any $k$-permissive $\mu$, if $Y\in \mathcal{C}_\mu$ and $M\equiv \mathcal{M}_\mu$ then for any $Z\cong X$ in $M$, $M$ contains $\mu(X,Y,g_M(Z))$ disjoint extensions over $Z$ each of the form of $Y/X$.
\end{obs}

We now endeavor to show that over any size of a base, there is an infinite recursive sequence of 3-unblockable extensions. Further, under the assumption of 3-permissiveness, each of these extensions is in $\C_{\mu}$.

\begin{lem}
\label{unblockability from permissiveness}
If $A\subseteq B$ is a minimally simply algebraic extension such that each element in $B\setminus A$ appears in at most $k$ edges in $B$, then $A\subseteq B$ is $k$-unblockable.
\end{lem}

\begin{proof}
Let $\mu$ be $k$-permissive such that $B\in \C_\mu$, let $Z\in\C_\mu$ contain $A$, and let $E=B\oplus_A Z$. Suppose $F,C^1,\ldots C^r$ are disjoint extensions of the form of a minimally simply algebraic extension $Y/X$ with $r>\mu(X,Y,g_{E}(F))$. By $k$-permissiveness of $\mu$, we have $r>k$. Thus, by Observation \ref{obs: if F is not in B_2 then r is restricted by valency} it must be that $F\subseteq Z$. Now we observe that every $C^i$ is either contained in $B\smallsetminus A$ or is contained in $Z$, for if it were partially but not totally in $B\smallsetminus A$, then we would have $\delta(C^i/Z)<0$ showing that $Z\not\leq E$, which is a contradiction. So, either $F\cup\bigcup_{i\leq r}C^i\subseteq Z$ or one of the $C^i$ is contained in $B\smallsetminus A$, which implies that $F\subseteq A$. Since $B$ is minimally simply algebraic over $A$, this implies $C^i=B\smallsetminus A$ and $F=A$. So, we conclude that there are already $\mu(X,Y,g_{E}(F))=\mu(X,Y,g_M(F))$ disjoint extensions of the form $Y/X$ over $F$ in $M$.
\end{proof}

\begin{lem}
	\label{unblockables over large base}
	There are infinitely many 2-unblockable minimally simply algebraic extensions over a set of size at least 3. Moreover, these extensions are in $\C_\mu$.
\end{lem}
\begin{proof}
	We define a sequence of $2$-unblockables over $\hat{A} = \set{a_1,\dots, a_k,g}$, where $k\geq 2$ and all the elements of $\hat{A}$ are distinct. For every $l>k$, by $a_l$ we mean $a_i$ where $i\equiv_k l$, $1\leq i\leq k$. Let $B=\set{b_1,\dots, b_{2t}}$ be new elements, where $t>k+1$ is arbitrary. Define $\hat{D}_t$ to be the hypergraph whose set of edges is
	\[
	R = \setcol{(b_i,a_i,b_{i+1})}{1\leq i< 2t}\cup\set{(b_{2t}, g, b_1}).
	\]
	Observe that $\hat{D}_t$ is in fact $D_t\setminus \{h\}$, so $\hat{D}_t\leq D_t$. Lemma \ref{if edges than D_t is classy} says that $D_t\in \C_\mu$, thus $\hat{D}_t\in \C_{\mu}$ as well.
	
	There is a clear bijection between elements of $B$ and edges in $R$, so $\delta(B/A) = 0$. By Lemma \ref{possible zeros in D_t}, there are no $A_0\subseteq \hat{A}$, $\emptyset\neq B_0\subset B$ such that $\delta(B_0/A_0) \leq 0$, so $B$ is minimally simply algebraic over $A$. Finally, 2-unblockability is immediate by Lemma \ref{unblockability from permissiveness}.
\end{proof}

Next, we will give constructions for infinitely many 3-unblockable extensions over sets of size $0$,$1$, or $2$.

\begin{defn}
	For $k\geq 3$ define on the set of vertices $\set{a_1,\dots, a_k}$ the following isomorphism types of minimal simple algebraicities:
	\begin{itemize}
		\item
		the \emph{ternary $k$-path} $P_k$ whose set of edges is
		\[
		\setcol{\set{a_i,a_{i+1},a_{i+2}}}{1\leq i \leq k-2}
		\]
		and whose base is $\set{a_1, a_k}$
		\item
		the \emph{ternary closed-$k$-path} $H_k$ whose set of edges is 
		\[
		\setcol{\set{a_i,a_{i+1},a_{i+2}}}{1\leq i \leq k-2}\cup\set{\set{a_{k-1},a_k,a_1}}
		\]
		and whose base is the singleton $\set{a_1}$
		\item
		the \emph{ternary $k$-loop} $L_k$ whose set of edges is
		\[
		\setcol{\set{a_i,a_{i+1},a_{i+2}}}{1\leq i \leq k-2}\cup\set{\set{a_{k-1},a_k,a_1},\set{a_k,a_1,a_2}}
		\]
		and whose base is $\emptyset$
	\end{itemize}
	Call any of these hypergraphs a \emph{generalized path}.
\end{defn}

\begin{lem}\label{kin3}
	Every generalized path $B$ over its base $A$ is 3-unblockable.
\end{lem}

\begin{proof}	
	By Lemma \ref{unblockability from permissiveness} we only need to show that $B$ is minimally simply algebraic over $A$. To see this, let $\emptyset\neq X\subset B$ and observe that there are strictly more elements in $X\setminus A$ than there are edges in $X$.
\end{proof}

\begin{lem}
	\label{msa in generalized path}
	If $A$ is a generalized path of size $k$ and $F\cup C\subset A$ are such that $C$ is minimally simply algebraic over $F$, then $C\cup F\cong P_l$ for some $l\leq k$, where $F$ is the pair of end-points of the path.
	
	Consequently, if $C/F$ is of the form of a minimally simply algebraic extension $Y/X$, then there can be at most one other extension in $A$ of the form of $Y/X$, and that is only in the cases that $A=L_{2l-2}$, or $F\cup C$ is an edge, i.e. $F\cup C \cong P_3$.
\end{lem}

\begin{proof}
	If $|C|=1$, then $C\cup F$ is an edge, i.e., isomorphic to $P_3$. Otherwise, each $a_i\in C$ must appear in at least two edges in $C\cup F$, hence $a_{i-1},a_{i+1}\in C\cup F$. Proceeding this way in both directions, since by assumption $F\cup C\neq A$, we find $a_m, a_M\in F$ distinct such that $a_j\in C$ for every $m<j<M$ (where if $j>k$ we replace it with $j-k$). By definition of minimal simple algebraicity, these are exactly the points in $F\cup C$, with $F=\set{a_m, a_M}$.
\end{proof}

\begin{cory}\label{pathshappen}
	If $\mu$ is 2-permissive, then $\C_\mu$ contains all generalized paths.
\end{cory}

Now Lemma \ref{unblockables over large base}, Lemma \ref{kin3}, and Corollary \ref{pathshappen} together yield:

\begin{cory}\label{unblockablesexist}
	There exists a recursive sequence of 3-unblockable extensions $X_{k,l}\subseteq Y_{k,l}$ with each $X_{k,l}$ having size $k$. Further, if $\mu$ is 3-permissive, then $Y_{k,l}\in \C_{\mu}$ for every $k,l$.
\end{cory}

\section{Constructing $T$}

In this section, we construct a theory $T:=T_{S_1}$ which is determined by a given r.e.\ set $S_1$. For any $A\subseteq \omega$, we denote by $A^{[i]}$ the $i$-th column of $A$, i.e., $A^{[i]} = \set{j\mid \langle i, j \rangle \in A}$. For the sake of uniformity of notation below, we assume that for every $i$, there are at least two numbers in $S_1^{[i]}$. For any r.e.\ set $S_1$, this construction will produce a strongly minimal theory $T$. In the next section, we will show that this theory $T$ (for any r.e.\ set $S_1$) is so that $\SRM(T)\supseteq [0,n]\cup \{\omega\}$. 

Given the r.e.\ set $S_1$, we define $S_0$ to the be the r.e. set which in each column $S_0^{[i]}$ contains all elements of $S_1^{[i]}$ except the last 2 enumerated into $S_1^{[i]}$. We assume that for each $i$, $S_1^{[i]}$ contains at least 2 elements (and the particular choice of $S_1$ we make below will have this property), so if $S_1^{[i]}$ is not infinite, then $S_1$ contains exactly two elements in the $i$th column which are not in $S_0$, which we call $\langle i, j_0\rangle, \langle i, j_1\rangle$, where $\langle i,j_0\rangle$ enters $S_1$ first. Otherwise, $S_1^{[i]}$ is infinite and $S_1^{[i]}= S_0^{[i]}$. In this case, the first two cases of the definition of $\mu$ below simply cannot hold since $\langle i,j_0\rangle$ and $\langle i,j_1\rangle$ are not defined.

Let $\L=\{R\}$ and $\L'=\{R\}\cup \{R_m\mid m\in \omega\}$ where $R$ is a ternary relation symbol and each $R_m$ is $n+2$-ary (the same $n$ as in $[0,n]\cup \set{\omega}$).

The outline of our construction is as follows: We first construct a theory $\hat{T}$ in the language $\hat{\mathcal{L}}=\{R\}\cup \{R_i\mid S_1^{[i]}\neq S_0^{[i]}\}$ via an amalgamation construction. We will then let $T$ be the reduct to the language $\mathcal{L}$. We will choose $\mu$ so that $\hat{T}$ is a definitional expansion of $T$. In particular, if $R_i$ holds on a tuple $\bar{a}$, then $\mu$ will allow extra extensions of some form over $\bar{a}$.

In building the recursive model $\mathcal{M}$ of dimension $\leq n$ or $\omega$, we work with the language $\mathcal{L}'$ and construct a model $\mathcal{M}'$ so that $\mathcal{M}=\mathcal{M}'\vert \mathcal{L}$. This is necessary since $\hat{L}$ is not recursive, so we don't know which relations are the important ones to consider. When we see a relation $R_i$ holding on a tuple, we are unsure if $R_i\in \hat{L}$, so the $\delta$-function might give a smaller value in $\mathcal{M}'$ than the correct dimension of the set in $\mathcal{M}$. Thus the dimension of $\mathcal{M}$ could be larger than the dimension of $\mathcal{M}'$. This poses no problem for the construction of the saturated model, since $\dim(\mathcal{M})\geq \dim(\mathcal{M}')=\omega$ ensures that $\mathcal{M}$ is saturated. But this poses a problem for constructing the recursive finite-dimensional models.

This is precisely the difficulty which we will use to ensure that $\SRM(T)\cap [n+1,\omega)=\emptyset$. In particular, if the enemy attempts to construct a model of dimension $k>n$ with basis $\bar{b}$, we will find some element $c$ so that $c\notin \acl(\bar{b})$. In particular, this $c$ will be some element which appears to satisfy a relation $R_i(\bar{b}_0,c)$ for $\bar{b}_0\subseteq \bar{b}$ of length $n+1$ and $R_i\notin \hat{L}$.

We choose $\mu$ to specifically help the recursive construction of models of dimension $\leq n$ in a way that does not help finite-dimensional models of larger dimension. In particular, we need a way to remove $R_i(\bar{a})$ from a tuple $\bar{a}$ if we suspect that $R_i\notin \hat{L}$. We will do this by defining $\mu$ so that having $g_C(\bar{a})$ small enough will allow an extra extension of some form over $\bar{a}$ in $C$. So there are two reasons a tuple might get an extra extension of this form: Either $R_i(\bar{a})$ holds or $g_C(\bar{a})$ is small enough. So if $g_C(\bar{a})$ is small enough we can remove the relation $R_i(\bar{a})$. We will ensure models of dimension $\leq n$ always have the opportunity to remove relations $R_i$ for $R_i\notin \hat{L}$, which we will be able to do since every $n+2$-tuple on which $R_i$ holds will have a finite $g$-value, but models of higher dimension will not have this opportunity.

\begin{defn}\label{def: limited away}
	We say that a relation symbol $R_i$ is ``limited away'' if $S_0^{[i]}=S_1^{[i]}$.
	
	Let $\hat{\L}=\{R\}\cup \{R_i\mid R_i\text{ is not limited away}\}$.
\end{defn}

We fix $\zeta$ to say that no relation holds on a subtuple of one where another relation holds (i.e. for each $U\neq V\in \hat{\L}$, if $U(\bar{x})$ and $\bar{y}\subseteq \bar{x}$, then $\neg V(\bar{y})$ holds). Note that $\zeta$ holds on each of the particular structures mentioned in Section \ref{Technicals}, as each of those use only the single ternary relation symbol $R$, thus the results about $\C_{\mu}$ in Section \ref{Technicals} hold for $\C_{\mu}^\zeta$ as well. Enumerate the relative quantifier-free types of infinitely many 3-unblockable extensions as built in section \ref{unblockthis} over a set of size $n+2$: $\langle \Omega_i\mid i\in \omega\rangle$. Note that these use only the relation symbol $R$.

We fix the function $\mu$ defined as follows:

\[\mu(A,B,m)= 
\begin{cases}
\abs{A}+4 &\text{ if $B/A$ is an $\Omega_{\langle i,j_0\rangle }$-extension and $R_i(A)$}\\
\abs{A}+4 &\text{ if $B/A$ is an $\Omega_{\langle i,j_1\rangle }$-extension, $R_i(A)$,} \\ & \text{ and $m\geq \langle i,j_1\rangle $}\\
\abs{A}+4 &\text{ if $B/A$ is an $\Omega_{\langle i,j\rangle }$-extension, $\langle i,j\rangle\in S_0$}\\
\abs{A}+3 &\text{ otherwise}
\end{cases}
\]

\begin{obs}
	$\mu$ is 3-permissive, so each $\Omega$-extension occurs the $\mu$-maximal number of times over any subset of a model of $\Th{\mathcal{M}_\mu^\zeta}$.
\end{obs}

Let $\hat{\mathcal{C}}$ be the class of $\hat{\mathcal{L}}$-structures $\mathcal{C}_\mu^\zeta$, and let $\mathcal{M}$ be the generic built from $\hat{\mathcal{C}}$. Let $\hat{T}$ be the theory of $\hat{\M}$. Let $\M$ be the reduct of $\hat{\M}$ to the language $\L$, and let $T$ be the theory of $\M$. It follows from the general construction that both $\hat{T}$ and $T$ are strongly minimal theories.

\begin{lem}\label{edef}
	Let $R_i$ be a relation symbol which is not limited away, i.e., $R_i\in \hat{\L}$. Then $\hat{T}\models R_i(\bar{x})\leftrightarrow \exists^{n+6}\bar{y}\,\Omega_{\langle i,j_0\rangle}(\bar{x},\bar{y})$.
	
	Thus $\hat{\M}$ is a definitional expansion of $\M$.
\end{lem}
\begin{proof}
	We first verify the leftward direction. In the definition of $\mu$, we see that 
	\[\mu(A,B,m)=
	\begin{cases}
	\abs{A}+4=n+6 &\text{ if $R_i(A)$}\\
	\abs{A}+3=n+5 &\text{ otherwise.}
	\end{cases}
	\]
	where $\tprqf(B/A)=\Omega_{\langle i,j_0\rangle}$.
	Thus, if $\neg R_i(A)$, then $\mu$ enforces that $\neg \exists^{n+6}\bar{y}\,\Omega_{\langle i,j_0\rangle}(\bar{x},\bar{y})$.
	
	If $R_i(\bar{x})$ holds, then since $\Omega_{\langle i,j_0\rangle}$ is an unblockable extension, in any model of $\hat{T}$, we have the maximal number of allowed extensions over any set, so we must have $\exists^{n+6}\bar{y}\,\Omega_{\langle i,j_0\rangle}(\bar{x},\bar{y})$.
\end{proof}

\begin{lem}\label{lem: tuple must extend to R_is}
	Let $R_i\in \hat{\L}$. For every $\bar{x}$ of size $n+1$ containing no tuple on which $R$ holds, there are exactly $n+4$ elements $y$ so that $\hat{\M}\models R_i(\bar{x},y)$.
\end{lem}
\begin{proof}
	We apply property (3'') of any model of $\hat{T}$. Let $B=\bar{x}$ and $C=\bar{x}y$ with the relative quantifier-free type consisting of the single relation $R_i(\bar{x},y)$. This is easily seen to be in $\C^\zeta_{\mu}$. Clearly, since $R_i$ is a symmetric $n+2$-ary relation, there is no set $Y$ with $\abs{Y}\leq |C\setminus B|= 1$ so that $\delta(Y\vert \{R_i\}/A\vert \{R_i\})<0$. Let $H',G'\subseteq C$ be of the form of a minimally simply algebraic $H/G$.
	Since the only relation that holds on $C$ is the single relation $R_i(\bar{x}y)$, we must have $H'$ is a singleton and $G'$ is the remainder. So, $H/G$ is not an $\Omega_i$-extension for any $i$ and  $\mu(G,H,g_C(G))=\mu(G,H,g_{\M\oplus_B C}(G))=\abs{G}+3$.  So (3'') guarantees that there are $\mu(\bar{x},\bar{x}y,g_\M(\bar{x}))=n+4$ realizations of this relative quantifier-free type over $\bar{x}$.
\end{proof}

\begin{cory}
	Each relation $R_i$ is both existentially and universally definable in $\M$.
\end{cory}
\begin{proof}
	That $R_i$ is existentially definable in $\M$ follows from Lemma \ref{edef}. Also, $\neg R_i(\bar{x},y)\leftrightarrow \exists^{n+4}z (R_i(\bar{x},z)\wedge z\neq y)\vee \bigvee_{u,v,w\in \bar{x}} R(u,v,w)$ gives an existential definition of $\neg R_i$.
\end{proof}

\begin{cory}\label{Tmodelcomplete}
	$T$ is model complete.
\end{cory}
\begin{proof}
	Let $\phi$ be an existential formula in $\hat{\L}$ written in prenex normal form. Then $\phi = \exists \bar{z} \bigvee\bigwedge Q(\bar{x})$ where $Q(\bar{x})$ is one of $R(\bar{x}),\neg R(\bar{x}), R_i(\bar{x}),\neg R_i(\bar{x})$. If $Q=R_i(\bar{x})$, then we replace $R_i$ by the existential formula in the language $\L=\{R\}$ which defines $R_i$. If $Q=\neg R_i(\bar{x})$, we replace $R_i$ by the universal formula in the language $\L$ which defines $R_i$. As such, we see that every existential $\hat{\L}$-formula is equivalent in $\hat{T}$ to an existential $\L$-formula. Model-completeness of $\hat{T}$ follows from Lemma \ref{modelcompletenessprereduct} and shows that every $\hat{\L}$-formula is equivalent over $\hat{T}$ to an existential $\hat{\L}$-formula, thus every formula is equivalent over $\hat{T}$ to an existential $\L$-formula.
	So every $\L$-formula is equivalent over $T$ to an existential $\L$-formula.
\end{proof}

\begin{lem}\label{Tflat}
	$T$ is flat and not disintegrated.
\end{lem}
\begin{proof}
	$\hat{T}$ is a definitional expansion of $T$, thus $\M$ and $\hat{\M}$ have the same acl-geometry. By Corollary \ref{firstTisflat}, the acl-geometry of $T$ is flat. To see that this geometry is non-disintegrated, consider $\set{a,b,c}\strong \M$ such that $\M\models R(a,b,c)$ and $R$ is the only relation holding on $\set{a,b,c}$. Then $c\in \acl(ab)$ but $c\notin \acl(a)\cup\acl(b)$.
\end{proof}

\section{Building the recursive models of $T$}

In this section, we construct a recursive copy of $\M$, the saturated model of $T$.  We will show that the $l$-dimensional submodels of $\M$ for $l\leq n$ are r.e. subsets, thus $\SRM(T)\supseteq [0,n]\cup\{\omega\}$. In the next section, we will choose $S_1$ to ensure that there are no other recursive models. We will have $S_1$ be the increasing union of the sets $\setcol{S_{1,s}}{s\in\omega}$. In our enumeration of $S$, we take $\abs{S_{1,s+1}\setminus S_{1,s}}=1$, so for each column $i$ and $s\in\omega$, the set $S_{1,s}^{[i]}$ is finite.

We will construct a copy of $\M$ where we also give uniformly $\Pi^0_1$ sets which are to represent the relations $R_i$. Thus, we may say we remove a relation $R_i$ from a tuple $\bar{a}$.

We construct the model in stages as usual by amalgamation: $N_0\subseteq N_1\subseteq N_2\subseteq \cdots\subseteq \bigcup_i N_i=\N$. Note that $N_{s-1}$ may have a relation $R_i$ hold on a tuple whereas $N_s$ removes that relation, but in the relation $R$, they are substructures. As usual, we say that a relation $R_i(\bar{a})$ holds in $\N$ if it holds on every structure in the chain where $\bar{a}\subseteq N_i$. We will ensure that for any tuple $\bar{a}$ in $N_k$, $\delta(\bar{a},N_k)=\delta(\bar{a},N_{k+1})$.
Furthermore, we will ensure that for every tuple $\bar{a}$, there is some $k$ so that the self-sufficient closure of $\bar{a}$ is the same (both in set and isomorphism-type) in every $N_l$ for $l\geq k$. 

At stage $s$, we consider the language $\L_s=\{R\}\cup \{R_i\mid i<s\}$. For every $i\leq s$, we let $\langle i,j^s_0\rangle, \langle i,j^s_1\rangle$ be the unique elements in the $i$th column of $S_{1,s}\smallsetminus S_{0,s}$ where $\langle i,j^s_0\rangle$ entered $S_1$ first, and we let

\[\mu_s(A,B,m)= 
\begin{cases}
\abs{A}+4 &\text{ if $B/A$ is an $\Omega_{\langle i,j^s_0\rangle }$-extension and $R_i(A)$}\\
\abs{A}+4 &\text{ if $B/A$ is an $\Omega_{\langle i,j^s_1\rangle }$-extension, $R_i(A)$,} \\ & \text{ and $m\geq \langle i,j^s_1\rangle $}\\
\abs{A}+4 &\text{ if $B/A$ is an $\Omega_{\langle i,j\rangle }$-extension, $\langle i,j\rangle\in S_{0,s}$}\\
\abs{A}+3 &\text{ otherwise}
\end{cases}
\]

Let $\C_s$ be the amalgamation class defined by this function $\mu_s$ and $\zeta$ (as defined below \ref{def: limited away}). We enumerate the amalgamation requirements for $\C_s$. A requirement is of the form: If $A\leq \N$ and $A\leq B\in \C_s$, then there is an $f:B\rightarrow \N$ which is the identity on $A$ so that $f(B)\leq \N$. We do this so that the order between two requirements which exist in $\C_s$ is preserved when considered in $\C_{s+1}$ and all new requirements from $\C_{s+1}$ appear after all requirements for amalgamations from $\C_s$ of sets of size $\leq s$ on the first $\leq s$ elements (i.e. the base is among the first $s$ elements of $\omega$ and the extension over the base is by at most $s$ new elements.). This ensures that every requirement in the full language is considered from some stage onward.

Now we describe the times when we might remove an occurrence of a relation.
\begin{defn}
	In a structure $C$, we say an occurrence of a relation $R_i(\bar{a})$ is defunct at stage $s$ if $g_C(\bar{a})<\langle i, j_1^s\rangle$.  
\end{defn}

In the definition of $\mu_s$ above, note that defunct relations are precisely the occurrences of $R_i(\bar{a})$ which do not allow $\abs{\bar{a}}+4$ $\Omega_{\langle i,j^s_1\rangle }$-extensions over $\bar{a}$. If $R_i(\bar{a})$ is defunct at stage $s$ and then another number gets enumerated into $S_1^{[i]}$, then nothing will prevent us from removing the relation $R_i$ from the tuple $\bar{a}$, with respect to future $\mu$ values, as seen in Lemma \ref{removedefunct}.

The following will be useful in the construction as we move from one stage to the next:

\begin{lem}
	\label{lem: s-1 subseteq s}
	$\C_{s-1}\subseteq \C_s$.
\end{lem}
\begin{proof}
	By inspecting the definitions of $\mu_s$, we see that $\mu_{s-1}\leq \mu_s$. This is because $\langle i,j^{s-1}_1\rangle$ is either $\langle i,j^{s}_1\rangle$ or $\langle i,j^s_0\rangle$ and $\mu_s$ can only increase by this. Similarly, $\langle i,j^{s-1}_0\rangle$ is either $\langle i,j^{s}_0\rangle$ or this number enters $S_0$. Thus $\C_{s-1}\subseteq \C_s$.
\end{proof}

\subsection*{Construction:}

At stage $s$ of the construction, we have built $N_{s-1}\in \C_{s-1}$ and we will construct $N_s\in \C_s$:

We first do the clean-up phase: Let $i$ be the unique number so that something is enumerated into $S_1^{[i]}$ at stage $s$. Let $\bar{a}$ be the smallest tuple (under a fixed ordering of  $\omega^{n+2}$ of order type $\omega$) on which $N_{s-1}$ has a defunct relation $R_i(\bar{a})$ at stage $s-1$. We will now remove $R_i$ from the tuple $\bar{a}$ and replace it with an extension involving only $R$ so that we maintain dimensions:

\begin{lem}\label{removedefunct}
	If $A\in \C_{s-1}$ and $R_i(\bar{a})$ is a defunct relation at stage $s-1$, then let $A_0$ be the result of removing $R_i(\bar{a})$ from $A$. Then there is an $A'\in \C_{s}$ containing $A_0$ so that for every $X\subseteq A$, $\delta(X,A)=\delta(X,A')$ and the only relation occurring on $A'$ outside of $A_0$ is the relation $R$. 
\end{lem}
\begin{proof}
	We first observe that $A_0\in \C_{s}$. Removing relations certainly does not cause the $\delta$-value of any set to drop below $0$ or make the structure fail to satisfy $\zeta$. So, we need only check that $A_0$ satisfies the $\mu_s$-bound. Suppose $F,C^1,\ldots C^r$ are disjoint subsets of $A_0$ and each $C^i/F$ is of the form of $Y/X$. Then by Observation \ref{removerelationsstayoftheform}, either these sets are each of the form of $Y/X$ in $A$ or else the relation removed is in $F$. In the former case, using Lemma \ref{lem: s-1 subseteq s} we know that $r\leq \mu_s(X,Y,g_A(F))$, and $g_{A_0}(F)\geq g_A(F)$. Since our $\mu_s$ is non-decreasing in the last coordinate, we see $r\leq \mu_s(X,Y,g_{A_0}(F))$. 
	
	Now we suppose the removed relation is in $F$. Let $Y'/X'$ be the minimally simply algebraic extension we get by adding the removed relation in $F$ to $X$ (so each $C^r/F$ is of the form of $Y'/X'$ in $A$). We have $g_{A}(F)=g_{A_0}(F)$ since no relations outside of $F$ have been removed. Unless $Y'/X'$ is an  $\Omega_{\langle i,j^{s}_0\rangle}$- or $\Omega_{\langle i,j^{s}_1\rangle}$-extension, we have $r\leq \mu_s(X',Y',g_A(F))=\mu_s(X,Y,g_{A_0}(F))$. 
	
	So, now we consider the critical case where $R_i(\bar{a})$ is the removed relation and the extensions $C^1,\ldots, C^r$ are $\Omega_{\langle i,j^s_0\rangle}$- or $\Omega_{\langle i,j^s_1\rangle}$-extensions over the base $\bar{a}$. For $\Omega_{\langle i,j^s_0\rangle}$, since $R_i(\bar{a})$ is defunct,  there can only be $n+5$ $\Omega_{\langle i,j^{s-1}_1\rangle}=\Omega_{\langle i,j^{s}_0\rangle}$-extensions in $A$, thus in $A_0$. Thus, there are $\leq \mu_s$-many even without $R_i(\bar{a})$. For $\Omega_{\langle i,j^s_1\rangle}$, since $\mu_{s-1}$ allowed only $n+5$ $\Omega_{\langle i, j^s_1\rangle}$-extensions over any base (in the fourth case of the definition of $\mu_{s-1}$), the $\mu_s$-bound is satisfied. In any case, $F,C^1,\ldots C^r$ does not violate the $\mu_s$-bound.
	
	Now, we apply Corollary \ref{droppingdimensionsatwill} to add an extension to $A_0$ to see that there is an $A'$ as needed: Corollary \ref{droppingdimensionsatwill} guarantees that, for an arbitrary $Z_0\subseteq A_0$, $\delta(Z_0,A')=\delta(Z_0,A_0)$ unless there is a $Y_0\supseteq Z_0$ witnessing $\delta(Z_0,A_0)$ with $\bar{a}\in Y$, in which case $\delta(Z_0,A')=\delta(Z_0,A_0)-1$. Letting $Z$ represent the set $Z_0$ in $A$, obtained by adding the removed relation $R_i(\bar{a})$, we wish to show that $\delta(Z_0, A') =\delta(Z,A)$. If there is a $Y_0\supseteq Z_0$ as described above, letting $Y\supseteq Z$ be obtained from $Y$ be adding the removed relation, we have $\delta(Y)=\delta(Y_0)-1$, so $\delta(Z,A)=\delta(Z_0,A_0)-1 = \delta(Z_0,A')$. If no such $Y_0$ exists, then clearly $\delta(Z,A) = \delta(Z_0,A_0) = \delta(Z_0,A')$. Putting these facts together, we see that in each case, $\delta(Z,A)=\delta(Z_0,A')$.
\end{proof}

By applying this Lemma to the defunct relation $R_i(\bar{a})$ we produce a structure $N_{s-1}'$ which is in $\C_{s}$. Then, we satisfy the first $s$ amalgamation requirements. To do this, we use the strong amalgamation lemma for the class $\C_s$ to construct $N_s\in \C_s$ which satisfies the first $s$ amalgamation requirements. This defines $N_{s}$ and the stage is done.

\begin{obs}
	$N_s\in \C_s$, thus we have maintained the inductive hypothesis for the next stage.
\end{obs}

Thus we have described the construction of a structure $\N$. We will show below that $\N\vert \L$ is a recursive presentation of the saturated model of $T$. From this, we will also produce recursive presentations of the models of dimension $\leq n$.

\subsection*{Verification:}

We now show that $\N\vert\L\cong \M$.

\begin{lem}\label{Ndefexp}
	For $R_i\in \hat{\L}$, $\N\models R_i(\bar{x})\leftrightarrow \exists^{n+6}\bar{y}\,\Omega_{\langle i,j_0\rangle}(\bar{x},\bar{y})$
\end{lem}
\begin{proof}
	Since $R_i$ is a $\Pi^0_1$ predicate on $\N$, if $\N\models R_i(\bar{x})$, then it does so from the first stage that $\bar{x}$ is first constructed. Thus, at every stage $s$ once $\langle i,j^s_0\rangle = \langle i,j_0\rangle$, we have $\mu_s(\bar{x},Y,g_{N_s}(\bar{x}))=n+6$ where $\tprqf(Y/\bar{x})=\Omega_{\langle i,j_0\rangle}$. Since $\Omega_{\langle i,j_0\rangle }$ is unblockable, we know that at some stage, when all the appropriate amalgamation requirements have been taken care of, we get $n+6$ many disjoint extensions of this form. But this is realized by the relation $R$ alone, so once it is satisfied, it is permanently satisfied inside $\N$.
	
	Suppose $\neg R_i(\bar{x})$ holds in $\N$. Then this is seen from some $s$ onward since $R_i$ is a $\Pi^0_1$ relation in $\N$. Since each $N_s\in \C_s$, we see that at no stage $\geq s$ can we have $n+6$-many $\Omega_{\langle i,j_0\rangle}$-extensions over $\bar{x}$, thus in the limit we do not have $n+6$-many $\Omega_{\langle i,j_0\rangle }$-extensions over $\bar{x}$.	
\end{proof}

We define $\hat{\N}$ to be the definitional expansion of $\N\vert \L$ to the language $\hat{\L}$ given by Lemma \ref{Ndefexp}. It now suffices to show that $\hat{\N}$ satisfies the properties of the generic of the class $\hat{\C}$ defined above. This proves that $\hat{\N}\cong \hat{\M}$, so $\N\vert \L = \hat{\N}\vert \L\cong \hat{\M}\vert \L = \M$.

Note that for $R_i\notin \hat{\L}$, we may very well have $R_i(\bar{a})$ holding in $\N$, but this will not be seen in $\N\vert \L$ or in $\hat{\N}$.

\begin{lem}\label{ssclosurestabilize}
	For every $\bar{a}\in \N$, there is some $k$ so that the self-sufficient closure of $\bar{a}$ is the same in every $N_l$ with $l\geq k$. 
\end{lem}
\begin{proof}
	The self-sufficient closure of $\bar{a}$ may change because some defunct relation $R_i$ in the self-sufficient closure is removed. At that point, we add more elements and occurrences of the relation $R$ so that the predimension of $\bar{a}$ is made the same in $N_{s}$ as in $N_{s-1}$. In doing so, the self-sufficient closure may have grown, but the total number of occurrences of relations other than $R$ on the self-sufficient closure has decreased. This can happen only finitely often.
\end{proof}

\begin{lem}\label{Csagree}
	Let $A$ be an $\hat{\L}$-structure. Then  $A\in \hat{\C}$ if and only if $A\in \C_s$ for all sufficiently large $s$.
\end{lem}
\begin{proof}
	
	Let $s$ be any stage large enough that for every relation $R_i\in \hat{\L}$ which occurs in $A$, $S_1^{[i]}$ is enumerated by stage $s$. It suffices to show that $A\in \hat{\C}$ if and only if $A\in \C_s$. Suppose $A\notin \hat{\C}$. This could happen if some set has $\delta(X)<0$ or violates $\zeta$, which certainly ensures that $A\notin \C_s$, or if $A$ violates the $\mu$-bound. But $\mu=\mu_s$ for any extension involving only the relations that occur in $A$, so this implies $A\notin \C_s$. The same argument shows the implication the other way.
\end{proof}

\begin{lem}
	$\hat{\N}$ is a generic model for the class $\hat{\C}$.
\end{lem}
\begin{proof}
	We need to verify the 3 facts:
	\begin{enumerate}
		\item $\hat{\N}$ is countable
		\item If $A\leq \hat{\N}$ is finite, then $A\in \hat{\C}$.

		\item Suppose $A\leq \hat{\N}$ and $A\leq B\in\hat{\C}$, then there is an embedding $f:B\rightarrow \hat{\N}$ so that $f(B)\leq \hat{\N}$ and $f$ is the identity on $A$.
		
	\end{enumerate}
	
	The first here is trivial, since $\hat{\N}$ is a definitional expansion of a countable structure $\N\vert \L$, so it is countable. Given $A\leq \hat{\N}$, $A\leq N_s$ and thus is in $\C_s$ for all sufficiently large $s$, by Lemma \ref{ssclosurestabilize}. Thus by Lemma \ref{Csagree}, $A\in \hat{\C}$.
	
	Lastly, consider $A\leq \hat{\N}$ and $A\leq B\in \hat{\C}$. Let $s$ be large enough that $A,B\in \C_t$ for every $t\geq s$. Further, let $s$ be large enough that $A\leq N_t$ for all $t\geq s$. Further, let $s$ be large enough that we consider the amalgamation requirement to build this $B$. Further, let $s$ be large enough that $S_{1,s}^{[i]}=S_1^{[i]}$ for every $R_i\in\hat{\L}$ occurring inside $B$. Then in $N_s$, we have the embedding $f:B\rightarrow N_s$ so that $f(B)\leq N_s$ and $f$ is the identity on $A$. Further, since no element is ever enumerated into $S_1^{[i]}$, we cannot ever remove any $R_i$-relations occurring in $B$. It suffices to see that $f(B)\leq N_t$ for every $t>s$. This follows from the following claim:
	
	\begin{claim*}
		If $X\leq N_s$ and no relation inside $X$ is ever removed, then $X\leq N_t$ for every $t>s$.
	\end{claim*}
\begin{proof}
	We proceed by induction. This is true for $t=s$. For every $t\geq s$, the $\delta$-value of $X$ as a substructure of $N_t$ is the same, since no relation is ever removed from $X$. Thus, we will unambiguously write $\delta(X)$ which does not depend on stage.
	
	When we pass from $N_{t-1}$ to $N_{t-1}'$ we ensured by adding elements and occurrences of $R$ that we have maintained dimension (i.e. $\delta(A,N_{t-1})=\delta(A,N_{t-1}')$ for every $A\subseteq N_{t-1}$), so $\delta(X)=\delta(X,N_{t-1})=\delta(X,N_{t-1}')$ and thus $X\leq N_{t-1}'$. Lastly, to go from $N_{t-1}'$ to $N_t$, we use strong amalgamation, so $X\leq N_{t-1}'\leq N_t$.
\end{proof}
\end{proof}

Thus, $\N\vert \L$ is a recursive copy of the saturated model of $T$.

\begin{lem}
	If $\bar{a}$ is an independent set in $\hat{\N}$ of size $\leq n$, then $\acl(\bar{a})$ is a $\Sigma^0_1$ subset of $\hat{\N}$.
\end{lem}
\begin{proof}
	Let $X$ be the set of elements which ever appear to be in $\acl(\bar{a})$. That is, we say $x$ is enumerated into $X$ at stage $s$ if there is some $Y\subseteq N_s$ so that $\bar{a}\leq Y$, $b\in Y$, and $\delta(Y/\bar{a})=0$. It is clear that $\acl(\bar{a})\subseteq X$ since for any $x\in \acl(\bar{a})$, such a $Y$ must exist in $\N$, thus in every large enough $N_s$. Also, $X$ is $\Sigma^0_1$. The fear is that since some relations $R_i$ get removed either due to being limited away (i.e. $R_i\notin \hat{\L}$) or in the clean-up phase, $X$ may contain some elements that are not actually algebraic over $\bar{a}$. 
	
	Fix $b\in X$. It enters $X$ because we see some $Y$ in $N_s$ containing $\bar{a}\cup \{b\}$ so that $\delta(Y/\bar{a})=0$. Any removed relation in the clean-up phase is immediately replaced by an $R$-witnessed dimension drop which maintains $\delta$. Thus, for every $t>s$, $\delta(Y,N_t)=\delta(Y,N_s)$, thus there is always some $Y'\subseteq N_t$ containing $Y$ with $\delta(Y')\leq \abs{\bar{a}}$. Let $Z$ be the self-sufficient closure of $\bar{a}b$ in $N_t$ for all sufficiently large $t$. We must have $\delta(Z)= \abs{\bar{a}}$, since $\delta(Z)=\delta(\bar{a}b,N_t)= \abs{\bar{a}}$. To see that $b\in \acl(\bar{a})$ witnessed by $Z$, we need only argue that each relation appearing in $Z$ is in $\hat{L}$. Were it true that $R_i(\bar{z})$ holds in $Z$ and $R_i\notin \hat{\mathcal{L}}$, then from some stage onwards the relation $R_i(\bar{z})$ would become defunct. This is because $R_i(\bar{z})$ implies $\delta(\bar{z})=\abs{\bar{z}}-1=n+1$ yet $\delta(Z) = |\bar{a}|\leq n$, so $g_{N_t}(\bar{z})$ is finite. But if $R_i(\bar{z})$ were to become defunct from some stage onwards, then it would eventually be removed.\footnote{Note that this is precisely the advantage that the models of dimension $\leq n$ have over models of finite dimension $>n$.}. But removing $R_i(\bar{z})$ would change the self-sufficient closure of $\bar{a}b$, contrary to the definition of $Z$. Thus, $b\in \acl(\bar{a})$ in $\hat{\N}$.
\end{proof}

\begin{cory}\label{compstuff}
	We conclude that $\SRM(T)\supseteq [0,n]\cup \{\omega\}$.
\end{cory}
\begin{proof}
	
	It is a standard fact that if $\mathcal{B}$ is any recursive structure and $\mathcal{A}$ is a recursively enumerable subset of $\mathcal{B}$, then $\mathcal{A}$ has a recursive presentation.
	
	Since the $k$-dimensional model of $T$ is $\acl(\bar{a})$ for an independent tuple in $\hat{\mathcal{N}}$ of size $k$ and $\mathcal{N}$ is recursive and infinite-dimensional, we conclude that $\SRM(T)\supseteq [0,n]\cup \{\omega\}$.
\end{proof}

\section{Defeating other models.}\label{choosingS1}
Next we ensure that $[n+1,\omega)\cap \SRM(T)=\emptyset$. We do this by enumerating $S_1$ appropriately. We will describe how numbers enter $S_1^{[i]}$ and we will note that we enumerate this column in order. In particular, at stage $s$, if we put anything into this column, we will put $s$ into $S_1^{[i]}$, and thus this is the largest number to enter this column. Even though in the construction of the recursive models, we assumed that $\abs{S_{1,s+1}\smallsetminus S_{1,s}}=1$ for all $s$, we will not be careful in this below, as any enumeration of an infinite $\Sigma^0_1$-set can be altered to give one enumerating the set in the same order and enumerating exactly one element per stage.

Let $B_i,\bar{b}_i$ be the $i$th pair consisting of a partial recursive atomic diagram of an $\L$-structure along with a finite tuple of size at least $n+1$ given by canonical index. We next describe a strategy to enumerate the $i$th column of $S_1$ so that either $B_i\not\models T$ or $\bar{b}_i$ is not a basis for $B_i$. These strategies are put together into a construction by running the first $s$ of these strategies in order at each stage $s$. At stage $s$, for $\bar{c}\in B_i$ we say that $R_i(\bar{c})$ holds if there is some $\langle i,j\rangle\in S_{1,s}\smallsetminus S_{0,s}$ so that $\exists^{n+6} \Omega_{\langle i,j\rangle}$-extensions over $\bar{c}$.

 We do this as follows:

Step 0: Let $\bar{b}$ consist of the first $n+1$ elements of $\bar{b}_i$. Enumerate $\langle i,0\rangle, \langle i,1\rangle $ into $S_1$. Wait until a stage where we see some element $c$ so that $R_i(\bar{b},c)$ holds.

Step 1: Let $j^s_0<j^s_1$ and $\{j^s_0,j^s_1 \}=(S_{1,s}^{[i]}\smallsetminus S_{0,s}^{[i]})$ at stage $s$. When we first come to this step, we define the set of obstructions to moving to the next step. If we see a set $Y\subseteq B_i$ and enough relations hold on $Y$ so that $\delta(Y)<\abs{\bar{b}_i}$ and $\bar{b}_i\subseteq Y$, then we call $Y$ an obstruction to moving to the next step. 

If $S_1^{[i']}$ enumerates a number after the stage when we entered this step, we call the relation symbol $R_{i'}$ suspicious (i.e., we suspect it might limit away). At each stage $t$, we let $\mathcal{L}_t$ be the set of non-suspicious relation symbols. We define $\dalet(Y)$ for any $Y$ to be $\abs{Y}-\Sigma_{Q\in \mathcal{L}_t}\#Q(Y)$. If at some stage $t$ we see $\dalet(Y)\geq \abs{\bar{b}_i}$, then we say the obstruction $Y$ has been removed.

We say that the requirement is ready for the next step if $\bar{b}c$ has $n+6$-many $\Omega_{\langle i,j^s_0\rangle}$ and $n+6$-many $\Omega_{\langle i,j^s_1\rangle}$-extensions and all obstructions have been removed. Wait until the requirement is ready for the next step. At that point, go to Step 2.

Step 2: Put $\langle i,s\rangle$ ($s$ is the current stage) into $S_{1,s}$. Note that this enumerates $\langle i,j^{s-1}_0\rangle$ into $S_{0,s}$. Return to Step 1.

\begin{lem}
	$[n+1,\omega)\cap \SRM(T)=\emptyset$.
\end{lem}
\begin{proof}
	We check that no $B_i,\bar{b}_i$ can be a recursive model of $T$ with basis $\bar{b}_i$.
	
	If we never leave step 0, then $S_1^{[i]}$ is finite, hence $R_i\in \hat{\L}$. By Lemma \ref{lem: tuple must extend to R_is}, clearly $B_i\not\models T$. So we may assume we have reached step 1 at some stage. Observe that in every subsequent stage, we have $R_i(\bar{b}c)$ holding. This is because we only enumerate a number into $S_{1,s}^{[i]}$ if we have both $n+6$-many $\Omega_{\langle i,j^{s-1}_0\rangle}$ and $n+6$-many $\Omega_{\langle i,j^{s-1}_1\rangle}$-extensions over $\bar{b}c$. Recall that $\Omega_{\langle i,j^{s-1}_1\rangle}$ only uses $R$, thus even after enumerating $\langle i,j^{s-1}_0\rangle$ into $S_0$, we still have $n+6$-many $\Omega_{\langle i,j^{s-1}_1\rangle}=\Omega_{\langle i,j^s_0\rangle}$-extensions over $\bar{b}c$, so $R_i(\bar{b}c)$ still holds.
	
	There are two possible outcomes to the strategy to defeat $B_i,\bar{b}_i$: Either we go through step 2 finitely or infinitely often.
	
	If we go through step 2 finitely often, then $R_i\in \hat{\L}$ and we must get stuck in step 1. This is either because of a non-removed obstruction, in which case $\delta(\bar{b}, B_i)<\abs{\bar{b}}$ or we never have both $n+6$-many  $\Omega_{\langle i,j_0\rangle}$ and $n+6$-many $\Omega_{\langle i,j_1\rangle}$-extensions over $\bar{b}c$. The first option means $\bar{B}_i$ is not a basis, so we may assume it is false, implying $\bar{b}c\leq B_i$. By (3'') and choice of $\mu$, since $\bar{b}c\leq B_i$, all of these $n+6$-many $\Omega_{\langle i,j_0\rangle}$ and $n+6$-many $\Omega_{\langle i,j_1\rangle}$-extensions over $\bar{b}c$ should exist. Thus, $B_i$ cannot model $T$. 
	
	In the infinite outcome, we argue that if $B_i\models T$ and $\bar{b}_i$ is independent in $B_i$, then $c\notin \acl(\bar{b}_i)$. Suppose otherwise that $B_i\models T$ and $\bar{b}_i$ is independent and there is some $Y$ containing $\bar{b}_ic$ and $\delta(Y)=\abs{\bar{b}_i}$. Recall $\delta$ is calculated in the language $\hat{\L}$, which in this case does not include $R_i$.
	Thus, when we also consider the relation $R_i(\bar{b}c)$, from some stage onward this $Y$ forms an obstruction that is never removed. 
	So, we get stuck in step 1 contradicting that we are in the infinite outcome. Thus, under the infinite outcome, if $B_i\models T$ and $\bar{b}_i$ is independent, then $c\notin \acl(\bar{b}_i)$, contradicting that $\bar{b}$ is a basis for $B_i$.
\end{proof}

\begin{thm}
	$T$ is a strongly minimal, flat, non-disintegrated, model complete theory in a language with finite signature, and $\SRM(T)=[0,n]\cup\{\omega\}$.
\end{thm}
\begin{proof}
	The theory $T$ is in the language $\L$ which has the finite signature $\{R\}$. In the previous lemma, we showed that $[n+1,\omega)\cap \SRM(T)=\emptyset$. In Corollary \ref{compstuff}, we showed that $[0,n]\cup \{\omega\}\subseteq \SRM(T)$. Thus $\SRM(T)=[0,n]\cup \{\omega\}$. In Corollary \ref{itisstronglyminimal}, we showed $\hat{T}$ is strongly minimal, which implies $T$ is strongly minimal. In Corollary \ref{Tmodelcomplete} we showed $T$ is model complete. In Lemma \ref{Tflat} we showed that $T$ is flat and non-disintegrated.
\end{proof}

\bibliographystyle{alpha}
\bibliography{refs}

\end{document}